\documentclass[11pt,reqno]{amsart}

\usepackage{amsmath, amsfonts, amsthm, amssymb,graphicx, color, cite,enumitem}
\usepackage[margin=1.2in]{geometry}

\usepackage{mathtools}
\mathtoolsset{showonlyrefs}   

\usepackage{hyperref}
\hypersetup{hidelinks}

\usepackage{amsmath}
\usepackage{amsfonts}%
\usepackage{amssymb}
\usepackage{mathrsfs}
\usepackage{titletoc}
\usepackage{enumerate}
\usepackage{cite}
\usepackage{color}
\allowdisplaybreaks
\date{\today}

 \numberwithin{equation}{section}

\newcommand{\dv}{\mathrm{div}\,}
\newcommand{\cl}{\mathrm{curl}\,}
\newtheorem{Theorem}{Theorem}[section]
\newtheorem{Lemma}{Lemma}[section]
\newtheorem{Proposition}{Proposition}[section]

\theoremstyle{definition}

\theoremstyle{remark}
\newtheorem{Remark}{Remark}[section]

\begin{document}
\title[a kinetic model]
 {Global existence for the relativistic Vlasov-Poisson system in a two-dimensional bounded domain} 

  \author[Y.-M. Mu ]{Yanmin Mu}
\address{Department of Applied Mathematics, Nanjing University of Finance \& Economics, Nanjing
 210046, China}
 \email{9120161021@nufe.edu.cn}

 \author[D. Wang]{Dehua Wang}
\address{Department of Mathematics,   University of Pittsburgh,  Pittsburgh, PA 15260, USA}
 \email{dhwang@pitt.edu}

\date{}

\begin{abstract}
In this paper, we prove the global existence of solutions to the relativistic Vlasov-Poisson system for general initial data in convex bounded domains of two space dimensions, assuming the specular reflection boundary conditions for the distribution density. The boundary conditions for the electric potential are considered in two cases: Neumann boundary conditions and homogeneous Dirichlet boundary conditions. The core ideas involve constructing suitable velocity lemmas and applying geometric techniques. In the two-dimensional case, it is crucial to select the arc length as the parameter of the curve and to further combine this with the Frenet-Serret formulas, enabling us to effectively describe the distribution density equation near the boundary and thus establishing a vital connection in the geometric representation. 
 \end{abstract}


\maketitle




\section{Introduction}

The relativistic Vlasov-Poisson system describes the collective dynamics of a collisionless plasma, where particles travel at nearly the speed of light and interact through their self-generated electric fields. The system consists of a relativistic Vlasov equation for the particle distribution function coupled with the Poisson equation for the electric potential.
In this paper, we are concerned with the global solution to the relativistic Vlasov-Poisson system (RVP) with general initial data in a convex, bounded domain $\Omega$ of two space dimensions:
 \begin{align}
&\partial_t f+\hat{v}\cdot\nabla_{x}f+\nabla_{x}\varphi\cdot\nabla_{v}f=0, 
\label{v1.5}\\
&\Delta\varphi=\rho, 
\label{v1.6}
\end{align}
where $x\in \Omega\subset \mathbb{R}^{2}$, $t>0$; $f=f(t,x,v)\geq 0$ denots  the  distribution density of particles at position $x$, time $t$, with momentum $v\in\mathbb{R}^{2}$;
$\hat{v}\in\mathbb{R}^{2}$ is the velocity that relates to the momentum $v$ according to Einstein's formula:
$$\hat{v}=\frac{v}{\sqrt{1+|v|^{2}}},$$
 $\varphi(t,x)$ is the electric potential, $\rho$  is the   charge  density   given by
 $$ \rho=\rho(t,x)= \int_{\mathbb{R}^{2}}f\, {\rm d}v,$$
 and the domain $\Omega$ is a  convex  bounded domain with $C^{5}$ boundary in $\mathbb{R}^{2}$. 
 We refer readers to \cite{KF,W} for more background on the relativistic Vlasov-Poisson system.

In the non-relativistic case, one has the following Vlasov-Poisson system (VP):
\begin{align}
&\partial_t f+v\cdot\nabla_{x}f+\nabla_{x}\varphi\cdot\nabla_{v}f=0,\label{v1.1}\\
&\Delta\varphi=\rho.\label{v1.2}
\end{align}
There have been many mathematical results for the non-relativistic Vlasov-Poisson system \eqref{v1.1}-\eqref{v1.2}, and we refer to Arsen'ev \cite{A}, Batt \cite{B}, Horst \cite{H}, Bardos-Degond \cite{BD}, Pfaffelmoser \cite{P}, Lions-Perthame \cite{LP} and the references therein for the global existence of solutions and related results.
It should be emphasized that  Pfaffelmoser \cite{P} and Lions-Perthame \cite{LP} demonstrated the global existence of general initial data by differential approaches. Nonetheless, it is well known that the existence of global solutions for large data is a challenging problem for the relativistic Vlasov-Poisson system. 
%
As indicated in \cite{GS}, in general, the well-known issues that emerge in the classical VP \eqref{v1.1}-\eqref{v1.2} do not arise here since $|\frac{d}{ds}x(s)|=|\hat{v}(s)|\leq 1$  along the characteristics for RVP \eqref{v1.5}-\eqref{v1.6}, which seems better than VP. However, from the energy equality of RVP in $\mathbb{R}^{3}$, 
$$ 
 \int\int_{\mathbb{R}^{6}}\sqrt{1+|v|^{2}}f(t,x,v){\rm d}x{\rm d}v+\frac{1}{2}\int_{\mathbb{R}^{3}}|\nabla \varphi|^{2}{\rm d}x=\text{constant}, \nonumber
$$
it follows that $\rho(t)\in L^{\frac{4}{3}}(\mathbb{R}^{3})$  for RVP,  while $\rho(t)\in L^{\frac{5}{3}}(\mathbb{R}^{3})$ for VP.
The primary challenge facing RVP at the moment is the loss of regularity for $\rho(t)$, which means that for general initial data, the global existence of classical  solutions  for RVP in $\mathbb{R}^{3}$  is still unknown.

For the Cauchy problem of the relativistic Vlasov-Poisson system  \eqref{v1.5}-\eqref{v1.6},  Glassey-Schaeffer\cite{GS,GS1}, and later Kiessling and Tahvildar-Zadeh \cite{KT2008} and Wang \cite{WangXC2023} established the spherically symmetric and  cylindrically symmetric solutions  in $\mathbb{R}^{3}$,
Rammaha\cite{R} proved the global existence for general initial data in $\mathbb{R}^{2}$, Had\v{z}i\'{c}-Rein \cite{H} obtained global existence and nonlinear stability.
For more results on related problems, we refer readers to \cite{GS2,GS3,GS4,GS5,GS6,GS7,GS8,K,WangXC2023} and their references. 
%
Nevertheless, there are no mathematical studies of well-posedness for the RVP solutions in the case of domains with boundaries. 
The motivation of this paper is to provide new insights into boundary-value problems in kinetic equations and to understand how boundaries influence the dynamics of RVP.


 For the distribution density, we consider the following initial and boundary conditions under which the distribution density exhibits specular reflection on the boundary:
   \begin{align}
&f(0,x,v)=f_{0}(x,v),\quad x\in \Omega,\; v\in \mathbb{R}^{2}, \label{v1.8}\\
&f(t,x,v)=f(t,x,v^{*}),\quad x\in \partial\Omega, \; v\in \mathbb{R}^{2},\; t>0, \label{v1.9}
\end{align}
 satisfying 
 \begin{align}
 &f_{0}(x,v)\geq 0,  \label{v1.10}\\
 &v^{*}=v-2\big(n_{x}\cdot v\big)n_{x}, \quad (x,v)\in \partial\Omega \times \mathbb{R}^{2}, \label{v1.11}
 \end{align}
 where $n_{x}$ denotes the outer normal vector at $x\in\partial\Omega$.
Meanwhile, regarding electric potential, we analyze two different types of boundary conditions. 
The first one is the Neumann boundary condition,
  \begin{align}
   &\frac{\partial \varphi}{\partial n_{x}}=h(x),\,\,\, x\in \partial \Omega,\,t>0,\label{v1.7}
   \end{align}
where the function $h$  is positive and satisfies the following compatibility condition:
\begin{align}
\int_{\Omega}f_{0}(x,v){\rm d}x{\rm d}v=\int_{\partial\Omega}h(x){\rm d}l, \label{v1.12}
\end{align}
and the second is the homogeneous Dirichlet boundary condition,
 \begin{align}
   & \varphi (t,x)=0,\,\,\, x\in \partial \Omega,\,t>0.\label{vD1.7}
 \end{align}
   
The well-posedness in bounded domains for VP has been extensively studied; see \cite{C,D,G1,G2,H1,HJV,HV,HV1,LM} and the references therein. 
Regarding the theory of well-posedness in bounded domains, many additional issues have emerged compared to the Cauchy problem of the Vlasov-Poisson system. Tracking the evolution of the characteristic curves associated with \eqref{v1.5}-\eqref{v1.6} is one of the challenges that must be addressed, for this purpose
in \cite{G2} Guo introduced the following ``singular sets",
\begin{align}
 \Gamma=\big\{(x,v)\in\Omega\times\mathbb{R}^{3}, x\in \partial\Omega,v\in T_{x}\partial\Omega\big\},   \label{v1.13}
\end{align}
 where $T_{x}\partial\Omega\subset \mathbb{R}^{3}$ is the tangent plane to $\partial\Omega$ at the point $x$.


\subsection{ A more convenient coordinate system near $\partial\Omega\times \mathbb{R}^{2}$}
 Assuming that $\partial\Omega$ is a smooth curve in $\mathbb{R}^{2}$ and its parametric equation can be expressed as $r(l)=(x_{1}(l),x_{2}(l))$, where $l$ represents the arc length parameter.
 At the point $r(l)$, we shall denote the outer normal to $\partial\Omega$  by $n(l)$.

The implicit function theorem shows that for $\delta > 0$ sufficiently small we can
parameterize uniquely the set of points  $x\in \partial\Omega + B_{\delta}(0) \subset \mathbb{R}^{2}  $ by the unique
values $(l, x_{\bot})$ satisfying the equation,
\begin{align}
x = r(l)-x_{\bot}n (l).  \label{v2.1}
\end{align}
Set
$$U(l)=\frac{dr(l)}{dl},\quad W(l)=\frac{dU(l)}{dl}; \quad \text{then}\,\;  n(l)=-\frac{W(l)}{|W(l)|}.$$
By the Frenet-Serret formulas, we can obtain a set of unit orthogonal local coordinate frames  $\big(U(l),n(l)\big)$, and then represent any vector $v\in  \mathbb{R}^{2}$ as,
\begin{align}
v = v_{\|}(l)-v_{\bot}n (l), \label{v2.2}
\end{align}
where $v_{\|}(l)=\omega U(l)\in T_{r(l)}(\partial\Omega), v_{\bot}\in \mathbb{R}.$

For the set of points  in the phase space $\Omega\times \mathbb{R}^{2}$ that are
close to the boundary $\partial\Omega\times \mathbb{R}^{2}$, we can denote $f(t,x,v)=f(t,l,x_{\bot},\omega,v_{\bot})$ by the system of coordinate $(l,x_{\bot},\omega,v_{\bot})$
and the equation of $f(t,x,v)$ will satisfy the following new form,
\begin{align}
f_{t}+\frac{1}{\sqrt{1+|v|^{2}}}\frac{\omega}{1-kx_{\bot}}\frac{\partial f}{\partial l}+\frac{v_{\bot}}{\sqrt{1+|v|^{2}}}\frac{\partial f}{\partial x_{\bot}}
+\sigma \frac{\partial f}{\partial \omega}+F\frac{\partial f}{\partial v_{\bot}}=0,\label{v2.3}
\end{align}
where $k>0$ is the curvature of boundary curves and
\begin{align}
&E=\nabla \varphi=E_{l}U(l)-E_{\bot}n(l), \quad E_{\bot}=-h, \label{v2.4}\\
&\sigma=E_{l}+\frac{v_{\bot}}{\sqrt{1+|v|^{2}}}\frac{k\omega}{1-kx_{\bot}},\quad F=E_{\bot}-\frac{1}{\sqrt{1+|v|^{2}}}\frac{k\omega^{2}}{1-kx_{\bot}}.\label{v2.5}
\end{align}
\begin{Remark}
The proof of \eqref{v2.3} is a standard change of variables  using the classical Frenet-Serret formulas; we will not provide details here.
\end{Remark}
\begin{Remark}
Note $1-kx_{\bot}>0$ for sufficiently small $x_{\bot}$. Since the domain $\Omega $ is convex, and due to $h>0$ we have
$F < 0$.
\end{Remark}

\subsection{Compatibility conditions and assumptions for the initial and boundary data}

In the process of establishing a classical solution, it is necessary that the initial data $f_{0}(x,v)$ satisfies the following compatibility conditions  at the reflection points of
$\partial\Omega\times \mathbb{R}^{2}$ (cf.  \cite{G2,H1}),
\begin{align}
&f_{0}(x,v)=f_{0}(x,v^{*}),\label{v2.6}\\
&v_{\bot}\Big[\nabla_{x}^{\bot}f_{0}(x,v^{*})+\nabla_{x}^{\bot}f_{0}(x,v)\Big]+2E_{\bot}(0,x)\nabla_{v}^{\bot}f_{0}(x,v)=0,\label{v2.7}
\end{align}
where $E_{\bot}(0, x)$ is the decomposition of the field $E (0, x)$ given by \eqref{v2.4} and
$\nabla_{x}^{\bot},\nabla_{v}^{\bot}$ are the normal components to $\partial\Omega$ of the gradients $\nabla_x,\nabla_v$ respectively.

We assume that the initial data $f_{0}(x,v)$ is constant near the singular set in order to establish a classical solution for general initial data (cf. \cite{HV}), as the characteristic curve continually   hits the boundary near the singular set, that is, 
the initial data $f_{0}(x,v)$  satisfies the following flatness condition near the singular set   
$\Gamma=\big\{(x,v)\in\Omega\times\mathbb{R}^{2}, x\in \partial\Omega,v\in T_{x}\partial\Omega\big\}$,     
\begin{align}
  f_{0}\in C^{1,\mu},\quad f_{0}(x,v)=\text{constant,\,\,  dist}((x,v), \Gamma)\leq \delta_{0},\label{v2.8}
\end{align}
for some $\delta_{0}>0$ small.
  
We note that if the function $h(t,x)$ depends on time and that $\frac{\partial h}{\partial t}$ is smooth enough, the    main result below still holds.
For convenience, we assume that $\frac{\partial h}{\partial t}=0$, that is,  $h=h(x)$.


\subsection{Main results}

We define some functional spaces as follows.  

For $\mu\in(0, 1), \nabla=(\nabla_{x},\nabla_{v})$,
  \begin{align}
  &\|f\|_{C^{1,\mu}(\bar{\Omega}\times \mathbb{R}^{2})}=\sup_{(x,v),(x',v')\in\bar{\Omega}\times \mathbb{R}^{2}}\Big(\frac{|\nabla f(x,v)-\nabla f(x',v')|}{|x-x'|^{\mu}+|v-v'|^{\mu}}\Big)
 +\|f\|_{L^{\infty}(\bar{\Omega}\times \mathbb{R}^{2})},\nonumber\\
  & \|f\|_{C^{1;1,\mu}_{t;x}([0,T]\times\bar{\Omega} )}=\sup_{x,x'\in\bar{\Omega},t,t'\in[0,T]}\frac{|\nabla_{x} f(t,x)-\nabla_{x} f(t',x')|}{|x-x'|^{\mu}}
  +\|f\|_{C([0,T]\times\bar{\Omega})} 
  +\|f_{t}\|_{C([0,T]\times\bar{\Omega})},\nonumber\\
  & \|f\|_{C^{1;1,\mu}_{t;(x,v)}([0,T]\times\Omega\times\mathbb{R}^{2})}\nonumber\\
  &=\sup_{x,x'\in\bar{\Omega},v,v'\in\mathbb{R}^{2}, t,t'\in[0,T]}
  \frac{|\nabla_{x} f(t,x,v)-\nabla_{x} f(t',x',v')|+|\nabla_{v} f(t,x,v)-\nabla_{v} f(t',x',v')|}{|x-x'|^{\mu}+|v-v'|^{\mu}}
  \nonumber\\
  &\qquad+\|f\|_{C([0,T]\times\bar{\Omega}\times\mathbb{R}^{2})}+\|f_{t}\|_{C([0,T]\times\bar{\Omega}\times\mathbb{R}^{2})},\nonumber
  \end{align}
 where $C([0,T]\times\bar{\Omega}),C([0,T]\times\bar{\Omega}\times\mathbb{R}^{2})$ are the spaces of continous functions bounded in the uniform norm, and
\begin{align}
 C&^{1,\mu}_{0}\Big(\bar{\Omega}\times \mathbb{R}^{2}\Big)\nonumber\\
 &=\Big\{f\in C^{1,\mu}\big(\bar{\Omega}\times \mathbb{R}^{2}\big):f \,\,\text{is \,compactly \,supported,}\;  \|f\|_{C^{1,\mu}(\bar{\Omega}\times \mathbb{R}^{2})}<\infty \Big\}.\nonumber
 \end{align}


We state  our main results of this paper as follows, with respect to the two different boundary conditions of the electric potential.
\begin{Theorem}\label{T1}
 Let $f_{0} \in C^{1,\mu}_{0}\big(\bar{\Omega}\times\mathbb{R}^{2}\big), f_{0}\geq 0$ for some $0 <\mu< 1$, satisfying
\eqref{v2.8}.  Suppose that $h \in C^{2,\mu} (\partial\Omega)$ satisfies \eqref{v1.12} and $h > 0$. Then there exists
a unique solution $f \in C^{1;1,\lambda}_{t;(x,v)} \big((0,\infty)\times \bar{\Omega}\times \mathbb{R}^{2}\big),
\varphi\in  C^{1;3,\lambda}_{t;x}\big((0,\infty)\times \bar{\Omega}\big) $
for
some $0 < \lambda < \mu$, of the relativistic Vlasov-Poisson system \eqref{v1.5}-\eqref{v1.6} with compact support
in $x$ and $v$, where the  initial boundary conditions of $(f,\varphi)$ satisfy \eqref{v1.8}-\eqref{v1.10} and \eqref{v1.7}  respectively.
\end{Theorem}

\begin{Theorem}\label{T02}
 Let $f_{0} \in C^{1,\mu}_{0}\big(\bar{\Omega}\times\mathbb{R}^{2}\big), f_{0}\geq 0$ for some $0 <\mu< 1$, satisfying
\eqref{v2.8}. Then there exists
a unique solution $f \in C^{1;1,\lambda}_{t;(x,v)} \big((0,\infty)\times \bar{\Omega}\times \mathbb{R}^{2}\big),
\varphi\in  C^{1;3,\lambda}_{t;x}\big((0,\infty)\times \bar{\Omega}\big) $
for
some $0 < \lambda < \mu$, of the relativistic Vlasov-Poisson system \eqref{v1.5}-\eqref{v1.6} with compact support
in $x$ and $v$, with the  initial boundary conditions of $(f,\varphi)$ satisfying \eqref{v1.8}-\eqref{v1.10} and \eqref{vD1.7}  respectively.
\end{Theorem}
\subsection{Difficulties and strategy of the proofs}

To prove the main results, we first apply the velocity lemma to establish the well-posedness of linearized problems. Then, we construct an iterative scheme and show the convergence of the iterative sequences. The main issues to address are the uniform boundedness in a given function space and the prolongation of uniform estimates for the functions ${f^{n}}$. Finally, we use bootstrapping techniques to reach the desired conclusions.

Now, we will discuss the major challenges. To this end, we first review the fundamental difficulties and core ideas of the initial boundary value problem from the perspective of the classic Vlasov-Poisson system. 
In \cite{G1},   the global existence was proved for the case of a half-space $\mathbb{R}^{3}_{+}$, assuming that the function $f_{0}$ remains constant near the singular set. This assumption  avoids the evolution of characteristic curves that are close to the singular set.
  In regions far away from the singular set, the number of collisions within a finite time interval can be bounded uniformly by using the velocity lemma method described in  
 \cite{LP}. This allows for a clear description of how the characteristic curves evolve. 
%
  In \cite{HV}, Hwang and Vel\'{a}zquez considered the Vlasov-Poisson system in a general bounded convex domain $\Omega\subset \mathbb{R}^{3}$, they addressed  the increasing complexity of the evolution of characteristic curves near the boundary, making it challenging to provide an accurate mathematical description. In order to establish global existence, the authors initially employed geometric methods, as outlined in \cite{HV}. Their results indicate that the geometric properties of the domain have a more significant influence on the characteristic curves than their dynamics.


However, understanding the evolution of the characteristic curves associated with the relativistic Vlasov-Poisson system \eqref{v1.5}-\eqref{v1.6} becomes extremely difficult near the singular set.
%
%
First, we face the complexities that arise within the Vlasov-Poisson system, including the challenges related to the behavior of characteristic curves near the singular set and the influence of regional boundaries and other factors. In this study, we will draw on insights from previous research on the Vlasov-Poisson system. Specifically, we assume that the initial data $f_{0}$ remains constant near the singular set. Additionally, to analyze the characteristic curves close to the boundary, we will use geometric methods. This approach is especially relevant for general bounded regions. The core idea of constructing suitable velocity lemmas and applying geometric techniques to prove the existence of global solutions for the relativistic Vlasov-Poisson system in convex bounded domains of two spatial dimensions remains valid. This applies to general initial data as well.

However, unlike the Vlasov-Poisson system, additional challenges arise with the relativistic Vlasov-Poisson system.  
In applying geometric methods to boundary issues, arc length plays a key role as a curve's characteristic in two dimensions. By using arc length as a parameter and considering the Frenet-Serret formulas, we can effectively describe the distribution density equation near the boundary $\partial\Omega$, thus establishing a vital connection in the geometric representation.

Moreover, identifying new coordinate variables is essential for developing appropriate velocity lemmas. These lemmas help describe scenarios where particles disperse from singular sets and face different numbers of collision barriers. In the context of the Newman boundary condition, we select coordinate variables $(\alpha,\beta)$ that satisfy the constraints previously mentioned,
\begin{align}
&\alpha(t,l,x_{\bot},\omega,v_{\bot})=\frac{v_{\bot}^{2}}{2}-L(t,l,0,\omega,v_{\bot})x_{\bot},\nonumber \\
&\beta(t, l,  x_{\bot}, \omega, v_{\bot})=2\pi F(t,l,x_{\bot},\omega,v_{\bot})+\pi (1-\frac{v_{\bot}}{\sqrt{2\alpha}}),\nonumber
\end{align}
where
\begin{align}
L(t,l,0,w,x_{\bot})=\sqrt{1+|v|^{2}}, \quad F(t,l,x_{\bot},\omega,v_{\bot})=-\sqrt{1+|v|^{2}}h(x)-k\omega^{2}.\nonumber
\end{align}
This variable  $\alpha$  characterizes the distance from points on the characteristic curve to the singular set. From this, the specific significance of selecting the tangential and normal directions of the regional boundary in local coordinates can be identified. The other variable  $\beta$, by contrast, describes the number of collisions between particles and the boundary. It is observed that the number of collisions is inversely proportional to the distance to the singular set. 
To ensure the regularity of the characteristic curve, the number of collisions must be bounded above uniformly within a certain time interval; it is for this purpose that assumption \eqref{v2.8} is proposed for the initial value.

For Dirichlet boundary conditions, deriving the velocity lemma is relatively more difficult. 
Additionally, choosing coordinate variables $\alpha$ becomes more complex.
$$\alpha=\frac{v_{\bot}^{2}}{2}-\varphi(t,x)-L(t,l,0,\omega,v_{\bot})x_{\bot}.$$
 It requires estimating the first derivatives of the electric potential  $\varphi$. This difference in handling the first derivative of  $\varphi$ is another notable difference between the two boundary types. The elliptical nature of the electric potential $\varphi$ in general domains, combined with the lack of a precise formula for $\varphi$, makes obtaining accurate estimates challenging. The core idea involves locally flattening the boundary and incorporating it with the Green's function in a half-space $\mathbb{R}^{2}_{+}$. Additionally, using revised boundary estimates for the  equation of $\varphi$ and constructing suitable supersolutions are important strategies.

Lastly, we found that the effect of regional boundaries is similar to the behavior of characteristic curves, which differs from the Vlasov-Poisson system. It is important to note that this study mainly focuses on a two-dimensional situation.

\subsection{Organization of the paper}

The structure of the paper is as follows.
From Section 2 to Section 5, we mostly address the Newmann boundary condition case.
In Section 2 we introduce  an iterative system and then  define a sequence of functions $\{ f^{n}\}$,  the
limit function as $ n \rightarrow \infty$ is the desired global solution of the RVP system.
In Section 3 we establish the well-posedness of the linear problem.
In Section 4 we show that the convergence of the iterative sequences $ \{ f^{n}\}$  under the condition that $Q (t)$ is bounded.
In Section 5 we prove the boundedness of the function $Q (t)$.  This concludes the proof of the first theorem.
We address the Dirichlet boundary condition in Section 6 and subsequently obtain the second result.

\bigskip

\section{Iterative Procedure}
We describe the iterative procedure in this section. The iterative sequence $(f^{n},\varphi^{n})$ satisfies the following system: 
\begin{align}
&f^{0}(t,x,v)=f_{0}(x,v),\; t\geq0,\; x\in\Omega,\; v\in \mathbb{R}^{2}, \label{v3.1}\\
&\partial_t f^{n}+\hat{v}\cdot\nabla_{x}f^{n}+\nabla_{x}\varphi^{n-1}\cdot\nabla_{v}f^{n}=0, \; x\in \Omega, \; v\in \mathbb{R}^{2}, \; t>0, \label{v3.2}\\
&\Delta\varphi^{n-1}=\rho^{n-1}=\int_{\mathbb{R}^{2}}f^{n-1}{\rm d}v, \; x\in \Omega, \; t>0, \label{v3.3}\\
&\frac{\partial \varphi^{n-1}}{\partial n_{x}}=h(x), \; x\in \partial \Omega, \; t>0,\label{v3.4}\\
&f^{n}(0,x,v)=f_{0}(x,v), \; x\in \Omega, \; v\in \mathbb{R}^{2}, \; t>0,\label{v3.5}\\
&f^{n}(t,x,v)=f^{n}(t,x,v^{*}), \; x\in \Omega,\;  v\in \mathbb{R}^{2}, \; t>0, \label{v3.6}
\end{align}
where $n=1,2,....$, and $f_{0},h$ satisfy \eqref{v1.10},\eqref{v1.12} and \eqref{v2.8}.
In the rest of this paper, we use the following notation,
\begin{align}
E^{n}=\nabla \varphi^{n}.  \label{v3.7}
\end{align}

The basic method for proving the global existence of solutions to RVP involves three key steps. 
First, one must construct an iterative sequence of approximate solutions denoted as $(f^{n},\varphi^{n})$. 
Second, it is necessary to prove the convergence of this sequence $(f^{n},\varphi^{n})$ as $n$ approaches infinity. 
Finally, one must show that the limit of the sequence indeed corresponds to solutions for the RVP.

\bigskip

\section{Linear Problem}

In order to show  the existence of the  iterative  sequence $f^{n}$,
 the well-posedness of the following linear problem 
 must be established.
 \begin{align}
&\partial_t f+\hat{v}\cdot\nabla_{x}f+ E\cdot\nabla_{v}f=0, \quad x\in \Omega\subset \mathbb{R}^{2}, v\in \mathbb{R}^{2},t>0, 
\label{vx3.1} \\
&f(t,x,v)=f(t,x,v^{*}),\quad x\in \partial\Omega, \; v\in \mathbb{R}^{2},\; t>0, \label{vx3.2} \\
 &\frac{\partial \varphi}{\partial n_{x}}=h(x),\,\,\, x\in \partial \Omega,\,t>0,
\label{vx3.3} 
\end{align}
Therefore,  we start by assuming that the given vector field $E=\nabla \varphi$
 satisfies suitable smoothness conditions.

The fundamental approach to solving the linear equation of the function $f$ is to take advantage of the method of characteristics.
To be more precise, the ordinary differential equations of the characteristic curve  $(X (s;t, x, v), V (s;t, x, v))$ are defined provided the   field $E=\nabla\varphi$ is given,
where $(x, v)\in \Omega\times\mathbb{R}^{2} $ by \eqref{v1.5}.
If $X\in\Omega$, we have
\begin{align}
&\frac{d X}{d s} = \hat{V}, \label{v4.1}\\
&\frac{d V}{d s}=E=\nabla_{x} \varphi, \label{v4.2}\\
&X(t;t, x, v)=x,\,\, V(t;t, x, v)=v.\label{v4.3}
\end{align}
If $X(s^{*};t,x,v)\in \partial\Omega$ at the time $s=s^{*}$,  with the help of velocity $V$ bouncing on the boundary, we can extend this definition of the characteristic equations to any time duration. that is,
\begin{align}
V((s^{*})^{+};t,x,v)&=\lim_{s>s^{*},s\rightarrow s^{*}}V(s;t,x,v) \nonumber\\
&=\Big(V\big((s^{*})^{-};t,x,v\big)\Big)^{*}=\Big(\lim_{s<s^{*},s\rightarrow s^{*}}V(s;t,x,v)\Big)^{*}, \label{v4.4}
\end{align}
where $(\cdot)^{*}$ is defined as in \eqref{v1.11}.

Now, we state the result  on the  linear problem as follows.

\begin{Theorem}\label{T2}
Assume that, given $T>0$, $E\in C^{0;1,\mu}_{t;x}([0,T]\times \bar{\Omega})$ for some $\mu\in (0,1)$, and $E\cdot n=h(x)>0$ on $\partial\Omega$.
 Suppose that $f_{0}\in C^{1,\mu}_{0}(\bar{\Omega}\times \mathbb{R}^{2}), f_{0}\geq 0$
 for some $\mu > 0$.
Then there exists a unique function $f \in C^{1;1,\lambda}_{t;x,v}([0,T]\times \Omega \times \mathbb{R}^{2})$, satisfying the linear
 relativistic
Vlasov-Poisson system \eqref{vx3.1}-\eqref{vx3.3} for some $0< \lambda< \mu$. Meanwhile,  the function $f$ satisfies,
\begin{align}
f(t,x,v)&\geq 0, \label{v4.5}\\
\int f(t,x,v){\rm d}x{\rm d}v&=\int f_{0}(x,v){\rm d}x{\rm d}v, \quad \forall \, t\in[0,T]. \label{v4.6}
\end{align}
\end{Theorem}

 The geometric coordinates $(t,l,x_{\bot},\omega, v_{\bot})$ mentioned above help us understand the characteristics of the equation near the boundary. Analyzing the behavior of the feature curve near the singular set, the frequency of feature curve collisions with the boundary, and other details also need further study. To do this, we introduce a new coordinate system, denoted by $(t,l,\alpha,\omega,\beta)$, to better describe these characteristics. In other words, the   new coordinates $(\alpha(t,l,x_{\bot},\omega,v_{\bot}), \beta(t, l,  x_{\bot}, \omega, v_{\bot}))$ are defined as follows,
\begin{align}
&\alpha(t,l,x_{\bot},\omega,v_{\bot})=\frac{v_{\bot}^{2}}{2}-L(t,l,0,\omega,v_{\bot})x_{\bot},\label{v4.7}\\
&\beta(t, l,  x_{\bot}, \omega, v_{\bot})=2\pi F(t,l,x_{\bot},\omega,v_{\bot})+\pi (1-\frac{v_{\bot}}{\sqrt{2\alpha}}),\label{v4.8}
\end{align}
where
\begin{align}
L(t,l,0,w,x_{\bot})=\sqrt{1+|v|^{2}}, \quad F(t,l,x_{\bot},\omega,v_{\bot})=-\sqrt{1+|v|^{2}}h(x)-k\omega^{2}.\nonumber
\end{align}

\begin{Remark}
The function $F(t,l,x_{\bot},\omega,v_{\bot})$ represents the number of collisions;  it will increase by one after each collision.  Therefore,  $F(t,l,x_{\bot},\omega,v_{\bot})$ is a step function mainly depending on  the independent variable t.  In the following, we will ignore the dependence on the variables $(l,x_{\bot},\omega,v_{\bot})$ and abbreviate it as $F(t)$.

\end{Remark}
\begin{Remark}
On the surface $\{\alpha = \rm constant\} $ where the trajectory lies,  $\beta $ is just a coordinate of the specific point.
The coefficient $2\pi $ 
in the definition of $\beta$ does not represent any specific meaning about angles.
\end{Remark}

\begin{Remark}
 At the moment of collision,  $x_{\bot}=0$ by \eqref{v4.7} and \eqref{v1.11},  $v_{\bot}$ immediately changes from $-\sqrt{2 \alpha}$ to $\sqrt{2 \alpha}$. Furthermore, by combining the definition of the function $F(t)$, we conclude that $\beta$ is continuous along characteristics.
\end{Remark}

 For convenience, we will abbreviate $\big(L(t,l,0,w,x_{\bot}),F(t,l,0,\omega,v_{\bot}),F(t,l,x_{\bot},\omega,v_{\bot})\big)$
 as $ \big(L(t,0),F(t,0),F(t,x_{\bot})\big)$.  We have, through calculation,
\begin{align}
x_{\bot}&=-\frac{\alpha}{L(t,0)}\Big[1-\big(1-\frac{\beta-2\pi H(t)}{\pi}\big)^{2}\Big],\nonumber\\
v_{\bot}&=\sqrt{2\alpha}\Big(1-\frac{\beta-2\pi H(t)}{\pi}\Big).\nonumber
\end{align}

By using the following coordinate transformation and \eqref{v4.7},\eqref{v4.8},
$$(t,l, x_{\bot},\omega,v_{\bot})\,\,\longrightarrow \,\,(t,l,\alpha,\omega,\beta),$$
 we can rewrite the equation \eqref{v2.3} as follows.
 
\begin{Lemma}\label{L1}
\begin{itemize}
  \item [(1)] Using the assumptions of the Theorem \ref{T1}, it follows that there exists a small number $\epsilon_{0}>0$, for $\forall t \in [0,t^{*}]$
  $$L(t,l,0,\omega)\leq -\epsilon_{0}< 0\quad on \;\partial\Omega.$$
  \item [(2)] In the new coordinate system $(t,l,\alpha,\omega,\beta)$, the system \eqref{v2.3} in $[\partial\Omega +B_{\delta}(0)]\times\mathbb{R}^{2}$ has the following form,
  \begin{align}
  f_{t}&+\frac{1}{\sqrt{1+|v|^{2}}}\frac{\omega}{1-kx_{\bot}}\frac{\partial f}{\partial l}+\sigma \frac{\partial f}{\partial \omega}
  +\Big[v_{\bot}\Big(F(t,x_{\bot})(1+\frac{x_{\bot}}{\sqrt{1+|v|^{2}}}h(x))\nonumber\\
  &-\frac{L(t,0)}{\sqrt{1+|v|^{2}}}\Big)
  -x_{\bot}\Big(\frac{1}{\sqrt{1+|v|^{2}}}\frac{\omega}{1-kx_{\bot}}\frac{\partial L(t,0)}{\partial l}+\sigma \frac{\partial L(t,0)}{\partial \omega}(t,0)\Big)\Big]\frac{\partial f}{\partial\alpha}\nonumber\\
  &+\Big[\pi \frac{2L(t,0)x_{\bot}}{(2\alpha)^{3/2}}F(t,x_{\bot})-\pi\frac{v^{2}_{\bot}}{(2\alpha)^{3/2}}\frac{L(t,0)}{\sqrt{1+|v|^{2}}}-\pi \frac{x_{\bot}v_{\bot}}{(2\alpha)^{\frac{3}{2}}}\Big(\sigma \frac{\partial L(t,0)}{\partial\omega}\nonumber\\
  &+ \frac{1}{\sqrt{1+|v|^{2}}}\frac{\omega}{1-k x_{\bot}}\frac{\partial L(t,0)}{\partial l}+F(t,x_{\bot})\frac{v_{\bot}h(x)}{\sqrt{1+|v|^{2}}}\Big)\Big]\frac{\partial f}{\partial\beta}=0. \label{v4.9}
  \end{align}
\end{itemize}
\end{Lemma}
\begin{proof}
(1) Since the bounded domain $\Omega$ is convex, then the term $k \omega^{2}$ is negative; meanwhile, $h(x)>0$ is continuous  on $\partial\Omega$ and thus we have
$$L(t,0)=-\sqrt{1+|v|^{2}}h(x)-k\omega^{2}< -h(x)\leq -\epsilon_{0}<0.$$
(2) By the defition of $\alpha, \beta$, we get
\begin{align}
&\frac{\partial \alpha}{\partial t}= \frac{\partial \beta}{\partial t}=0, \nonumber\\
&\frac{\partial \alpha}{\partial l}=-\frac{\partial L(t,0)}{\partial l}x_{\bot},\quad\qquad  \qquad\frac{\partial \beta}{\partial l}=-\frac{\pi x_{\bot}v_{\bot}}{(2\alpha)^{3/2}}\frac{\partial L(t,0)}{\partial l},\nonumber\\
&\frac{\partial \alpha}{\partial x_{\bot}}=- L(t,0),\,\,\qquad \qquad \qquad\frac{\partial \beta}{\partial x_{\bot}}=-\frac{\pi v_{\bot}}{(2\alpha)^{3/2}} L(t,0),\nonumber\\
&\frac{\partial \alpha}{\partial \omega}=-\frac{\partial L(t,0)}{\partial \omega}x_{\bot},\quad \quad \qquad\quad\frac{\partial \beta}{\partial \omega}=-\frac{\pi x_{\bot}v_{\bot}}{(2\alpha)^{3/2}}\frac{\partial L(t,0)}{\partial \omega},\nonumber\\
&\frac{\partial \alpha}{\partial v_{\bot}}=v_{\bot}+\frac{v_{\bot} x_{\bot}}{\sqrt{1+|v|^{2}}}h(x),\quad  \frac{\partial \beta}{\partial v_{\bot}}
=\frac{2\pi v_{\bot}L(t,0)}{(2\alpha)^{3/2}} +\frac{\pi v^{2}_{\bot}x_{\bot}}{(2\alpha)^{3/2}}\frac{h(x)}{\sqrt{1+|v|^{2}}}.\nonumber
\end{align}
From the above equalities, it is straightforward to derive the equation \eqref{v4.9}.  
\end{proof}

\begin{Remark}
The equation \eqref{v4.9} can be reformulated as follow,
\begin{align}
  f_{t}&+\frac{1}{\sqrt{1+|v|^{2}}}\frac{\omega}{1-kx_{\bot}}\frac{\partial f}{\partial l}+\sigma \frac{\partial f}{\partial \omega}
  +\Big[v_{\bot}\Big(F(t,x_{\bot})-F(t,0)\nonumber\\
  &+\frac{x_{\bot}h(x)}{\sqrt{1+|v|^{2}}}F(t,x_{\bot})\Big)
  -x_{\bot}\Big(\frac{1}{\sqrt{1+|v|^{2}}}\frac{\omega}{1-kx_{\bot}}\frac{\partial L(t,0)}{\partial l}+\sigma \frac{\partial L(t,0)}{\partial \omega}(t,0)\Big)\Big]\frac{\partial f}{\partial\alpha}\nonumber\\
  &+\Big[-\frac{\pi}{\sqrt{2\alpha}}F(t,0)+\pi\frac{2x_{\bot}L(t,0)}{(2\alpha)^{3/2}}\big(F(t,x_{\bot})-F(t,0)\big)-\pi \frac{x_{\bot}v_{\bot}}{(2\alpha)^{\frac{3}{2}}}
 \Big(\sigma \frac{\partial L(t,0)}{\partial\omega}+\nonumber\\
   &+\frac{1}{\sqrt{1+|v|^{2}}}\frac{\omega}{1-k x_{\bot}}
  \frac{\partial L(t,0)}{\partial l}+F(t,x_{\bot})\frac{v_{\bot}h(x)}{\sqrt{1+|v|^{2}}}\Big)\Big]\frac{\partial f}{\partial\beta}=0. \label{v4.10}
  \end{align}
\end{Remark}

\begin{Remark}
We point out that in a singular set,
the dynamic equations of the tangential part to $\partial\Omega$ of the characteristics are
\begin{align}
\frac{d l}{d s}=\frac{\omega}{\sqrt{1+|v|^{2}}},\quad \frac{d \omega}{d s}=E_{l}.\nonumber
\end{align}
\end{Remark}

\begin{Remark} For  the new variables $(l, \alpha, \omega, \beta)$, the initial data $f_{0}$ satisfies the following two conclusions,
\begin{itemize}
  \item [(1)]
   Combining compatibility conditions \eqref{v2.6},\eqref{v2.7}, we have that 
 $f_{0}$ is still a $C^{1,\mu}$ functionin in the variables $(l, \alpha, \omega, \beta)$.
  \item [(2)] The flatness condition \eqref{v2.8} yields that $f_{0}$ is constant for $0\leq\alpha\leq C\delta_{0}$ for some $C > 0 $.
  In fact, since $ dist((x,v), \Gamma)\leq \delta_{0}$, it follows that $|x_{\bot}|+|v_{\bot}|\leq \delta_{0}$, and with the help of the compact support of $f_{0}$  for the variable v and the defition of $\alpha$ we obtain that $0\leq\alpha\leq C\delta_{0}$.
\end{itemize}
\end{Remark}

\subsection{Velocity lemma}

To establish the well-posedness of linear problems, a thorough analysis of the characteristic curves near the singular set is crucial, especially when considering collisions with boundaries. This analysis requires applying the velocity lemma discussed in this section. 
In \cite{LP}, the authors first introduced the velocity lemma to solve the Cauchy problem for VP system, while in \cite{HV}  the corresponding velocity lemma was established in a bounded domain.
 The velocity lemma states that the number of boundary collisions is finite if the characteristic curve is sufficiently distant from the singular set. This ensures the regularity of the characteristic curve and is an important prerequisite for deriving classical solutions for the RVP system.

\begin{Lemma}\label{L2}
For a given constant $\delta >0$, let $\Gamma_{\delta}=([\partial\Omega +B_{\delta}(0)]\cap \Omega)\times \mathbb{R}^{2}.$
Suppose that the regularity assumptions of $E$ in Theorem \ref{T2} hold. Then
the existence of solutions to the characteristic equations \eqref{v4.1}-\eqref{v4.3} can be obtained in $[0,T]$
for any $(x,v)\in \bar{\Omega}\times \mathbb{R}^{2}$.
Furthermore,    the following estimate applies  for any $(x, v) \in\Gamma_{\delta}$,
\begin{align}
C_{1}\Big(v^{2}_{\bot}(0)+x_{\bot}(0)\Big)\leq \Big(v^{2}_{\bot}(t)+x_{\bot}(t)\Big) \leq C_{2}\Big(v^{2}_{\bot}(0)+x_{\bot}(0)\Big), \; t\in[0,T],\label{v4.11}
\end{align}
for some   positive constants $C_{1}, C_{2}$ depending only on $T,f_{0},\|E\|_{L^{\infty}([0,T],C^{\frac{1}{2}}(\Omega))}$.
\end{Lemma}
\begin{proof}
On the one hand, due to $L(t,0)\leq -\epsilon_{0}$ in Lemma \ref{L1} and \eqref{v4.7}, we have $$\alpha (t)\geq C(v^{2}_{\bot}+x_{\bot})$$ for some constant $C$.
On the other hand, by the boundedness of velocity $|v|$ (see  Section 6) as well as the continuity of function $h(x)$, it follows that 
$$\alpha (t)\leq C(T)(v^{2}_{\bot}+x_{\bot}).$$
Therefore we say that these two quantities $\alpha(t), v^{2}_{\bot}+x_{\bot}$ are equivalent.
Because of the boundedness of the domain, we only need to deal with the case in which $v^{2}_{\bot}+x_{\bot}$ is small, that is, the point $(x,v)$ is near the singular set.
Along  the characteristic curves, $\alpha (t)$ satisfies the following equation by \eqref{v4.10},
\begin{align}
\frac{d \alpha}{d t}=&v_{\bot}\Big(F(t,x_{\bot})-F(t,0)\Big)
  +\frac{v_{\bot}x_{\bot}h(x)}{\sqrt{1+|v|^{2}}}F(t,x_{\bot})\nonumber\\
  &-x_{\bot}\Big(\frac{1}{\sqrt{1+|v|^{2}}}\frac{\omega}{1-kx_{\bot}}\frac{\partial L(t,0)}{\partial l}+\sigma \frac{\partial L(t,0)}{\partial \omega}\Big)\nonumber\\
  =:&M_{1}+M_{2}+M_{3}.\nonumber
\end{align}
Next, we calculate $M_{i}$ separately as,
\begin{align}
|M_{1}| &\leq |v_{\bot}|\cdot\big|E_{\bot}(t,l,x_{\bot})-E_{\bot}(t,l,0)\big|+|v_{\bot}|\cdot\left|\frac{k \omega^{2}}{\sqrt{1+|v|^{2}}}\frac{kx_{\bot}}{1-kx_{\bot}}\right|\nonumber\\
&\leq \|E\|_{L^{\infty}\big([0,T],C^{\frac{1}{2}}(\Omega)\big)}\|v_{\bot}\|x^{\frac{1}{2}}_{\bot}+C(T)x_{\bot},\nonumber\\
|M_{2}| &\leq  \|h\|_{L^{\infty}}\|F\|_{L^{\infty}} x_{\bot}\nonumber\\
&\leq  \|h\|_{L^{\infty}}\big(\|E\|_{L^{\infty}}+C(T)\big) x_{\bot}\leq C(T) x_{\bot}. \nonumber
\end{align}
By the definition $L(t,0)$, we have
$$\frac{\partial L(t,0)}{\partial l}=-\sqrt{1+|v|^{2}}\nabla h(x)\cdot U(l)-\frac{d k}{dl}\omega^{2},\qquad
\frac{\partial L(t,0)}{\partial \omega}=-\frac{\omega}{\sqrt{1+|v|^{2}}}h(x)-2k \omega. $$
Furthermore,
\begin{align}
|M_{3}| &\leq |\frac{\omega}{1-kx_{\bot}}\frac{\partial L(t,0)}{\partial l}|x_{\bot}+|\sigma \frac{\partial L(t,0)}{\partial \omega}(t,0)|x_{\bot}\nonumber\\
 &\leq C(T)x_{\bot}.\nonumber
\end{align}
Therefore, we conclude the following estimate
\begin{align}
\Big|\frac{d \alpha}{d t}\Big|\leq C\|E\|_{L^{\infty}\big([0,T],C^{\frac{1}{2}}(\Omega)\big)}\|v_{\bot}\|x^{\frac{1}{2}}_{\bot}+C(T)x_{\bot},\nonumber
\end{align}
where $C$ depends only on $h$ and the geometric properties of $\partial\Omega$. Meanwhile,  $\alpha$ and $v^{2}_{\bot}+x_{\bot}$ are equivalent, so we have
\begin{align}
\Big|\frac{d \alpha}{d t}\Big|\leq C \alpha. \label{v4.12}
\end{align}
Thus we obtain
\begin{align}
C_{1}\alpha(0)\leq \alpha(t)\leq C_{2} \alpha(0), \nonumber
\end{align}
where $C_{1}, C_{2}$ depend  on
$T,f_{0},\|E\|_{L^{\infty}([0,T],C^{\frac{1}{2}}(\Omega))}$, that is, \eqref{v4.11} holds.
\end{proof}

\subsection{Well-posedness of the linear problem}

The primary focus of this subsection is to prove Theorem \ref{T2}. The fundamental approach involves integrating the linear equation along the characteristics. Nonetheless, it is essential to pay close attention to the regularity of characteristic curves in conjunction with reflection boundary conditions since characteristic curves frequently collide with boundaries  $\partial\Omega$.

\begin{proof}[Proof of Theorem \ref{T2}]
According to the characteristic equation \eqref{v4.1}-\eqref{v4.3},   we can define the following function,
\begin{align}
f(t,x,v)=f_{0}(X(0;t,x,v),V(0;t,x,v)).\label{v4.13}
\end{align}
We will show that this function is the solution to the linear problem satisfying Theorem \ref{T2}.

{\bf Step 1.} The characteristic curves determined by the  equations\eqref{v4.1}-\eqref{v4.3} for $(x,v)\in \Omega \times \mathbb{R}^{2}$  have two basic facts.
\smallskip

\noindent (1) The curves never intersect with the singular set.
  
  In fact, for the curves $(X(s;t,x,v),V(s;t,x,v)), 0\leq s\leq t$ with $(X(t;t,x,v)=x,V(t;t,x,v)$ $=v)\in \Omega \times \mathbb{R}^{2} $, one has 
  $\alpha(t)>0$. By Lemma \ref{L2}, $\alpha (s)$ maintains upper and lower bounds during the time period $[0, t]$, then it follows that $\alpha (s)> 0$ for any $s\in[0,t]$. 
  If the curves intersect with the singular set at some time  $s=s_{0}$, we have $\alpha(s_{0})=0$, which  contradicts with $\alpha (s_{0})> 0$.
  
  \smallskip
  
\noindent (2)  The curves intersect the boundary $\partial\Omega\times \mathbb{R}^{2}$ at most a finite number of times.
  
  {\em Case 1.} If the characteristics starts in the region $\alpha(0)\leq C\delta_{0}$,  by Lemma \ref{L2}, $\alpha (s)\leq C_{1}(T)\delta_{0}, \forall s\in [0,T]$ for some $C_{1}(T)$,  so because  $f_{0}$ is constant, we have $f= \text{constant}$.
  
  {\em Case 2.} If the characteristics starts in the region $\alpha(0)\geq C\delta_{0}$,  by Lemma \ref{L2}, $\alpha (s)\geq C_{2}(T)\delta_{0}, \forall s\in [0,T]$ for some $C_{2}(T)$,  we claim that $\Big|\frac{d \beta}{ d t}\Big|\leq \frac{C}{\sqrt{\delta_{0}}}$, and further that the upper bound of the number of collisions does not exceed $\frac{C}{\sqrt{\delta_{0}}}$.
  In fact, by \eqref{v4.10},
  \begin{align}
\frac{d \beta }{d t}=&-\frac{\pi}{\sqrt{2\alpha}}F(t,0)+\pi\frac{2x_{\bot}L(t,0)}{(2\alpha)^{3/2}}\big(F(t,x_{\bot})-F(t,0)\big)-\pi \frac{x_{\bot}v_{\bot}}{(2\alpha)^{\frac{3}{2}}}
 \Big(\sigma \frac{\partial L(t,0)}{\partial\omega}\nonumber\\
   &+\frac{1}{\sqrt{1+|v|^{2}}}\frac{\omega}{1-k x_{\bot}}
  \frac{\partial L(t,0)}{\partial l}+F(t,x_{\bot})\frac{v_{\bot}h(x)}{\sqrt{1+|v|^{2}}}\Big).\nonumber
 \end{align}
 By direct calculation, it follows that
 \begin{align}
&\Big|\frac{\pi}{\sqrt{2\alpha}}F(t,0)\Big|\leq \frac{C(T)}{\sqrt{2\alpha}},\nonumber\\
&\Big|\pi\frac{2x_{\bot}L(t,0)}{(2\alpha)^{3/2}}\big(F(t,x_{\bot})-F(t,0)\big)\Big|\leq \frac{C(T)}{\sqrt{2\alpha}},\nonumber\\
&\Big|\pi \frac{x_{\bot}v_{\bot}}{(2\alpha)^{\frac{3}{2}}}
 \Big(\frac{1}{\sqrt{1+|v|^{2}}}
   \frac{\omega}{1-k x_{\bot}}
  \frac{\partial L(t,0)}{\partial l}+\sigma \frac{\partial L(t,0)}{\partial\omega}+F(t,x_{\bot})\frac{v_{\bot}h(x)}{\sqrt{1+|v|^{2}}}\Big)\Big|
  \leq C(T).\nonumber
 \end{align}
Therefore, we have $\Big|\frac{d \beta}{ d t}\Big|\leq \frac{C}{\sqrt{\delta_{0}}}.$

Let the back-time cycle from $(t,x,v)$ be $$(t^{l},x^{l},v^{l})=(t,x,v),(t^{l-1},x^{l-1},v^{l-1}),...,(t^{1},x^{1},v^{1}), (0,x_{0},v_{0}).$$
By \eqref{v2.3}, we get
\begin{align}
&\frac{d x_{\bot}(s)}{d s}=\frac{v_{\bot}(s)}{\sqrt{1+|v|^{2}}},\nonumber\\
&\frac{d v_{\bot}(s)}{d s}= F(s),\nonumber
\end{align}
and moreover,
\begin{align}
v_{\bot}(s)&=v^{*}_{\bot}(t^{j})+\int^{s}_{t^{j}}F(\tau){\rm d}\tau, \,\, t^{j-1}< s<t^{j}, \nonumber\\
x_{\bot}(t^{j})&=x_{\bot}(t^{j-1})+\int^{t^{j}}_{t^{j-1}}\frac{v_{\bot}(s)}{\sqrt{1+|v|^{2}}}{\rm d}s\nonumber\\
&=x_{\bot}(t^{j-1})+\int^{t^{j}}_{t^{j-1}}\frac{v^{*}_{\bot}(t^{j})}{\sqrt{1+|v|^{2}}}{\rm d}s
+\int^{t^{j}}_{t^{j-1}}\int^{s}_{t^{j}}\frac{F(\tau)}{\sqrt{1+|v(s)|^{2}}}{\rm d}\tau{\rm d}s.\nonumber
\end{align}
Since $x_{\bot}(t^{j})=x_{\bot}(t^{j-1})=0$, we obtain
\begin{align}
&\Big|\int^{t^{j}}_{t^{j-1}}\frac{v^{*}_{\bot}(t^{j})}{\sqrt{1+|v|^{2}}}{\rm d}s\Big|
=\Big|\int^{t^{j}}_{t^{j-1}}\int^{s}_{t^{j}}\frac{F(\tau)}{\sqrt{1+|v(s)|^{2}}}{\rm d}\tau{\rm d}s\Big|,\nonumber\\
&\Big|\int^{t^{j}}_{t^{j-1}}\int^{s}_{t^{j}}\frac{F(\tau)}{\sqrt{1+|v(s)|^{2}}}{\rm d}\tau{\rm d}s\Big|\leq
C(T)\big(t^{j}-t^{j-1}\big)^{2},\nonumber\\
&\Big|\int^{t^{j}}_{t^{j-1}}\frac{v^{*}_{\bot}(t^{j})}{\sqrt{1+|v|^{2}}}{\rm d}s\Big|\geq C(T)|v^{*}_{\bot}(t^{j})||t^{j}-t^{j-1}|,\nonumber
\end{align}
also due to Lemma \ref{T2}, it follows that $|v^{2}_{\bot}(t^{i})|\geq C(T)\delta_{0}$,  then we get
$|t^{j}-t^{j-1}|\geq C(T)$, which implies that the number of collisions in the time period [0, T]  is finite.
According to $\Big|\frac{d \beta}{ d t}\Big|\leq \frac{C}{\sqrt{\delta_{0}}}$, we have
\begin{align}
&\beta(t^{j})-\beta(t^{j-1})\leq \frac{C}{\sqrt{\delta_{0}}}(t^{j}-t^{j-1}),\nonumber\\
&\beta(t^{j})=2\pi H(t^{j}_{+}),\,\,\beta(t^{j-1})=2\pi H(t^{j-1}_{+}),\nonumber\\
&H(t^{j}_{+})=H(t^{j-1}_{+})+1.\nonumber
\end{align}
As a result, we sum the above equalities over $j$ to obtain,
\begin{align}
2\pi l \leq \sum^{l}_{j=1}\frac{C}{\sqrt{\delta_{0}}}(t^{j}-t^{j-1})\leq \frac{C}{\sqrt{\delta_{0}}}T. \nonumber
\end{align}


{\bf Step 2.} We now analyze the regularity of the characteristic curves.\\
If the characteristic curves do not collide with the boundary, by $E \in C^{1,\mu}_{x}$, based on the classical regularity estimates for the solutions of ODEs, we can introduce the functions $X(s; t,x,v),V(s;t,x,v)$  that are $C^{1,\mu}$ with respect to the variables $(x,v)$.

If the characteristic curves intersect with the boundary $\partial\Omega $ at time $s=s(t,x,v)$, for simplicity, letting $s$  be the first collision time on  the back-time cycle from $(t,x,v)$ to $(0,x_{0},v_{0})$, then,  by \eqref{v4.1}-\eqref{v4.3}, one has
\begin{align}
&V(s^{+};t,x,v)=v+\int^{s}_{t}E(\tau, X(\tau;t,x,v)){\rm d}\tau,\nonumber\\
&X(s;t,x,v)=x+\int^{s}_{t}\frac{v}{\sqrt{1+|v(\tau)|^{2}}}{\rm d}\tau+\int^{s}_{t}\int^{\xi}_{t}\frac{E(\tau,X(\tau;t,x,v))}{\sqrt{1+|V(\xi)|^{2}}}{\rm d}\tau{\rm d}\xi,\nonumber\\
&V(s^{-};t,x,v)=V(s^{+};t,x,v)-2\big(V(s^{+};t,x,v)\cdot n\big)n.\nonumber
\end{align}
From the above equalities and the regular assumption on $\partial\Omega$, it follows that  $X(s; t,x,v)$ and $V(s;t,x,v)$   are also  $C^{1,\mu}_{(x,v)}$.
Thus, for the  characteristic curves of finite bounces, the functions $X(s; t,x,v)$ and $V(s;t,x,v)$  consist of a finite number of piecewise $C^{1,\mu}_{(x,v)}$ functions.
Hence, combined with \eqref{v4.13}, we have $ f \in C^{1,\mu}_{(x,v)}$.
 Finally, let's briefly summarize the above results as
\begin{itemize}
  \item [(1)]If $\alpha \leq C \delta_{0}$, $f\equiv \text{constant}$.
  \item [(2)]If $\alpha \geq C \delta_{0}$, $f \in C^{1,\mu}_{(x,v)}$.
  \item [(3)]$f_{t}=-v\cdot \nabla_{x}f-E\cdot\nabla_{v}f.$
\end{itemize}
 We therefore prove that $f(t,x,v) \in C^{1;1,\lambda}_{t,(x,v)}\big([0,T]\times\Omega\times\mathbb{R}^{2}\big)$ for some $0< \lambda< \mu$.

{\bf Step 3.} Uniqueness.

According to the theory of ODEs,
we know that the solutions of the characteristic equations $X(0; t,x,v),V(0;t,x,v)$ are unique, thus the function $f(t,x,v)=f_{0}(X(0;t,x,v), V(0;t,x,v))$
is also unique. Therefore, we have completed the proof of Theorem \ref{T2}.
\end{proof}

\section{On the Convergence of the Sequence \{$f^{n}$\}}

\subsection{The solution of the Poisson equation with Neumann boundary conditions}
In this subsection, we recall some results on the solution of the Poisson equation with the Neumann boundary conditions.
\begin{Proposition}[\!\!\cite{E}]\label{P1}
Given a bounded domain $\Omega \subset \mathbb{R}^{2}$ with a smooth boundary $\partial\Omega,$ for the following Poisson equation with Neumann boundary conditions:
\begin{align}
\triangle \varphi&=\rho (x),  x\in \Omega,  \label{v5.1}\\
\frac{\partial \varphi}{\partial n}&= h(x), x\in\partial\Omega, \label{v5.2}\\
\int_{\Omega}\rho(x){\rm d}x&=\int_{\partial\Omega}h(x){\rm d}l, \label{v5.3}
\end{align}
  there exists a Green's function $G(x,y)$ such that we have
\begin{align}
\varphi(x)=\int_{\Omega}G(x,y)\rho(y){\rm d}y-\int_{\partial\Omega}G(x,y)h(y){\rm d}l.\label{v5.4}
\end{align}
For the Green's function, we have the following estimates:
\begin{align}
|\nabla_{x}G(x,y)|\leq \frac{C}{|x-y|},\,\,|\nabla^{2}_{x}G(x,y)|\leq \frac{C}{|x-y|^{2}},\quad x,y\in\bar{\Omega}. \label{v5.5}
\end{align}
where $C$ depends only on the domain $\Omega$.
\end{Proposition}

\begin{Proposition}[\!\!\cite{N}]\label{P2}
Assume that $\rho(x)\in C^{0,\alpha}(\bar{\Omega}), h(x)\in C^{1,\alpha}(\bar{\Omega})$ for $\alpha\in (0,1)$, then there exists
a solution $\varphi\in C^{2,\alpha}(\bar{\Omega})$ (unique up to an additive constant) to
the problem \eqref{v5.1}-\eqref{v5.3}, such that
\begin{align}
\Big\|\varphi-\frac{1}{|\Omega|}\int_{\Omega}\varphi\Big\|_{C^{2,\alpha}(\Omega)}\leq C\Big(\|\rho\|_{C^{0,\alpha}}+\|h\|_{C^{1,\alpha}}\Big), \label{v5.6}
\end{align}
with $C=C(\Omega,\alpha)$.
\end{Proposition}

\subsection{The iterative sequence \{$f^{n}$\} is globally defined in time}
For the solution  $f(t,x,v)\in C^{1;1,\lambda}_{t;(x,v)}\big([0,T]\times\Omega\times\mathbb{R}^{2}\big)$ of  the linear problem, we need to proceed further to estimate
$E=\nabla \varphi,\, \rho(x)=\int_{\mathbb{R}^{2}}f(t,x,v){\rm d}v, \,F(t,x)$ and so on. Let's start with some notations.
For a function $g:\Omega\rightarrow \mathbb{R}$, we define the seminorm $[\cdot]_{0,\lambda;x}$ as
$$[g]_{0,\lambda;x}\equiv\sup_{x,y\in\Omega}\frac{|g(x)-g(y)|}{|x-y|^{\lambda}}.$$
For the function $f(t,x,v)$, the support of $|v|$ in the time period $[0,t]$ is defined as
\begin{align}
 Q(t)=\sup\{|v|: \,(x,v)\in \text{supp } f(s), \; 0\leq s\leq t\}.\label{v5.7}
\end{align}
We have the following proposition.

\begin{Proposition}\label{P3}
Assume that $\varphi$ satisfies \eqref{v1.6},\eqref{v1.7} and $f(t,x,v)\in C^{1;1,\lambda}_{t;(x,v)}\big([0,T] \times\Omega\times\mathbb{R}^{2}\big)$, $E$ satisfies the regularity hypothesis in Theorem \ref{T2}. Then we have the following estimates,
\begin{align}
|\rho(t,x)|&\leq C(T)\|f\|_{C^{1;1,\lambda}_{t;(x,v)}\big([0,T]\times\Omega\times\mathbb{R}^{2}\big)},\,\,(t,x)\in[0,T]\times\Omega,\label{v5.8}\\
|\nabla\rho(t,x)|&\leq\int_{\mathbb{R}^{2}}|\nabla _{x}f(t,x,v)|{\rm d}v \nonumber\\
&\leq C(T)\|f\|_{C^{1;1,\lambda}_{t;(x,v)}\big([0,T]\times\Omega\times\mathbb{R}^{2}\big)},\,\,(t,x)\in[0,T]\times\Omega,\label{v5.9}\\
|\rho_{t}(t,x)|&\leq C(T)\|f\|_{C^{1;1,\lambda}_{t;(x,v)}\big([0,T]\times\Omega\times\mathbb{R}^{2}\big)},\,\,(t,x)\in[0,T]\times\Omega,\label{v5.10}\\
|F(t,x)|&\leq C(T),\quad (t,x)\in[0,T]\times\partial\Omega,\label{v5.11}\\
|E(t,x)|&\leq C(T)\Big(\|f\|_{C^{1;1,\lambda}_{t;(x,v)}\big([0,T]\times\Omega\times\mathbb{R}^{2}\big)}+1\Big),\,\,(t,x)\in[0,T]\times\Omega,\label{v5.12}\\
|\nabla E(t,x)|&+[\nabla E(t,\cdot)]_{0,\lambda;x}+|\nabla^{2}E(t,x)|+[\nabla^{2} E(t,\cdot)]_{0,\lambda;x}\nonumber\\
&\leq C(T)\Big(\|f\|_{C^{1;1,\lambda}_{t;(x,v)}\big([0,T]\times\Omega\times\mathbb{R}^{2}\big)}+1\Big),\,\,(t,x)\in[0,T]\times\Omega,\label{v5.13}\\
|E_{t}(t,x)|&\leq C(T)\|f\|_{C^{1;1,\lambda}_{t;(x,v)}\big([0,T]\times\Omega\times\mathbb{R}^{2}\big)},\,\,(t,x)\in[0,T]\times\Omega.\label{v5.14}
\end{align}
\end{Proposition}
\begin{proof}
Firstly, we give the estimates about $\rho(t,x)$ as follows:
\begin{align}
|\rho(t,x)|&= \Big|\int_{\mathbb{R}^{2}}f(t,x,v){\rm d}v\Big|\leq \|f\|_{L^{\infty}}Q^{2}(T)\nonumber\\
&\leq C(T)\|f\|_{C^{1;1,\lambda}_{t;(x,v)}\big([0,T]\times\Omega\times\mathbb{R}^{2}\big)},\nonumber\\
|\nabla \rho(t,x)|&= \Big|\int_{\mathbb{R}^{2}}\nabla f(t,x,v){\rm d}v\Big|\leq \|\nabla f\|_{L^{\infty}}Q^{2}(T)\nonumber\\
&\leq C(T)\|f\|_{C^{1;1,\lambda}_{t;(x,v)}\big([0,T]\times\Omega\times\mathbb{R}^{2}\big)},\nonumber\\
|\rho_{t}(t,x)|&= \Big|\int_{\mathbb{R}^{2}}f_{t}(t,x,v){\rm d}v\Big|\leq \|f_{t}\|_{L^{\infty}}Q^{2}(T)\nonumber\\
&\leq C(T)\|f\|_{C^{1;1,\lambda}_{t;(x,v)}\big([0,T]\times\Omega\times\mathbb{R}^{2}\big)},\nonumber
\end{align}
thus, \eqref{v5.8}-\eqref{v5.10} hold.

Sencondly, since $x\in\partial\Omega$,\,\,i.e. $x_{\bot}=0$, then
\begin{align}
|F(t,x)|=|-h(x)-\frac{k \omega^{2}}{\sqrt{1+|v|^{2}}}|
\leq C(T), \nonumber
\end{align}
so, \eqref{v5.11} is true.

Finally, we deal with $E(t,x)$. Due to \eqref{v5.6}, and $E=\nabla \varphi$,
\begin{align}
|E(t,x)|+|\nabla E(t,x)|+[\nabla E(t,\cdot)]_{0,\lambda;x}&\leq  C\Big(\|\rho\|_{C^{0,\lambda}}+\|h\|_{C^{1,\lambda}}\Big)\nonumber\\
&\leq  C(T)\Big(\|f\|_{C^{1;1,\lambda}_{t;(x,v)}\big([0,T]\times\Omega\times\mathbb{R}^{2}\big)}+1\Big).\nonumber
\end{align}
Since $f(t,x,v)\in C^{1;1,\lambda}_{t;(x,v)}\big([0,T] \times\Omega\times\mathbb{R}^{2}\big)$ and $E=\nabla\varphi$, it follows that
$\nabla \rho \in  C^{0,\lambda}(\Omega)$ and $E$ satisfies \eqref{v1.6}-\eqref{v1.7}. Once again applying \eqref{v5.6}, we obtain
\begin{align}
|\nabla^{2}E(t,x)|+[\nabla^{2} E(t,\cdot)]_{0,\lambda;x}&\leq  C\Big(\|\nabla \rho\|_{C^{0,\lambda}}+\|\nabla h\|_{C^{1,\lambda}}\Big)\nonumber\\
&\leq  C(T)\Big(\|f\|_{C^{1;1,\lambda}_{t;(x,v)}\big([0,T]\times\Omega\times\mathbb{R}^{2}\big)}+1\Big).\nonumber
\end{align}
According to the assumptions on $\varphi$, the function $\varphi_{t}$ satisfies
\begin{align}
&\Delta\varphi_{t}=\rho_{t}, \,\,x\in \Omega,t>0, \nonumber\\
&\frac{\partial \varphi_{t}}{\partial n_{x}}=0,\,\, x\in \partial \Omega,t>0,\nonumber
\end{align}
and by \eqref{v5.6},
$$|E_{t}(t,x)|\leq\|\rho_{t}\|_{C^{0,\lambda}} \leq C(T)\|f\|_{C^{1;1,\lambda}_{t;(x,v)}\big([0,T]\times\Omega\times\mathbb{R}^{2}\big)}.$$
Then  the inequalities \eqref{v5.12}-\eqref{v5.14} follow.
\end{proof}
\begin{Proposition}\label{P4}
Let  $ 0<\lambda<\mu$,
 suppose that $f_{0}\in C^{1,\mu}_{0}(\bar{\Omega}\times \mathbb{R}^{2}), f_{0}\geq 0$ satisfies \eqref{v2.8} and $h \in C^{1,\mu}(\partial\Omega),\,\,h(x)>0$.
Then the  iterative sequence $f^{n}$ is globally defined for each $x\in \Omega,\,v\in\mathbb{R}^{2},\,0\leq t<\infty $.
 Moreover,  the function $f^{n}(t,x,v)\in C^{1;1,\lambda}_{t;(x,v)}\big([0,T] \times\Omega\times\mathbb{R}^{2}\big)$  for $\forall\, T >0$ satisfies
\begin{align}
\|f^{n}\|_{L^{\infty}}&=\|f_{0}\|_{L^{\infty}} ,\label{v5.15}\\
\int \rho^{n}(t,x){\rm d}x&=\int f_{0}(x,v){\rm d}x{\rm d}v.  \label{v5.16}
\end{align}
\end{Proposition}
\begin{proof}
We use induction to prove this proposition.

{\bf Step 1.} If $n=1$, we estimate the support of the function $f^{1}$.
According to \eqref{v5.12},we have
\begin{align}
\Big|\frac{d V}{d s}\Big|=|E^{0}|=|\nabla \varphi^{0}|\leq C(T)\Big(\|f_{0}\|_{C^{1,\mu}}+1\Big)\leq C(T). \nonumber
\end{align}
Thus,
\begin{align}
|V(s)|\leq C(T)(1+s), \,\, i.e. \, |v|\leq C(T)(1+t). \nonumber
\end{align}
Then  we have the following estimate on $\rho^{1}(t,x)$,
\begin{align}
|\rho^{1}(t,x)|=\Big|\int_{\mathbb{R}^{2}}f^{1}(t,x,v){\rm d}v\Big|\leq \|f^{1}\|_{L^{\infty}} C(T)(1+t)^{2}\leq C(T)(1+t)^{2}.\nonumber
\end{align}
 In the light of $f_{0}\in C^{1,\mu}_{0}(\bar{\Omega}\times \mathbb{R}^{2})$ as well as Proposition \ref{P3}
which imply $E^{0}\in C^{1;1,\mu}_{t,x}([0,T]\times \bar{\Omega})$, using Therorem\ref{T2}, we  obtain $f^{1}\in C^{1;1,\mu}_{t,(x,v)}([0,T]\times\bar{\Omega}\times\mathbb{R}^{2})$ and
$$ \int \rho^{1}(t,x){\rm d}x=\int f_{0}(x,v){\rm d}x{\rm d}v.$$

{\bf Step 2.} If $n=2$,
by means of Proposition \ref{P3} and the boundedness of the domain $\Omega$ , it follows that
$$\|E^{1}\|_{C^{1;1,\mu}_{t,x}([0,T]\times \bar{\Omega})}\leq C(T)\Big(\|f^{1}\|_{C^{1;1,\lambda}_{t,(x,v)}}+1\Big).$$
Applying Theorem \ref{T2}, we can obtain $f^{2}$ is well defined in $C^{1,\lambda}_{x,v}$ for $ t\in [0,+\infty).$
Furthermore, based on the estimates of $\rho^{1}(t,x)$ and  the explicit formula of the Poisson equation \eqref{v5.4}, 
by direct calculation we see that $|E^{1}(t,x)|$  can be controlled by some function $g(t)$ which is a continuous increasing function with respect to t.
From the characteristic equation, we can determine the support set of $f^{2}$ on velocity v.
Applying Theorem \ref{T2} again, we infer that $f^{2}$ is globally defined for each $x\in \Omega,\,v\in\mathbb{R}^{2},\,0\leq t<\infty $,
and the function $f^{2}(t,x,v)\in C^{1;1,\lambda}_{t;(x,v)}\big([0,T] \times\Omega\times\mathbb{R}^{2}\big)$  is bounded.
By repeating the above argument, we can conclude that there is also the same conclusion for the sequence $f^{n}$.

{\bf Step 3.}
Because  $f^{n}$ is defined based on the propagation of characteristics, it follows that $\|f^{n}\|_{L^{\infty}}=\|f_{0}\|_{L^{\infty}}$, i.e., \eqref{v5.16} holds.
The  equality \eqref{v5.16} can be obtained by integrating the equation \eqref{v3.2} over the space $\Omega\times\mathbb{R}^{2}$ with regard to the variables $(x,v)$.
 We complete the proof.
\end{proof}

\subsection{Convergence of the iterative sequences  \{$f^{n}$\} }
 Under the condition that  the support of  $f^{n}(t,x,v)$ with respect to $v$ is uniformly bounded, the iterative sequence $f^{n}(t,x,v)$ will converge to
 the solution of the RVP system. We denote
 the support of  $f^{n}$ on $v$  as $Q^{n}$,
\begin{align}
 Q^{n}(t)=\sup\{|v|: \,(x,v)\in \text{supp } f^{n}(s), 0\leq s\leq t\}.\label{v5.17}
\end{align}

We start with several auxiliary lemmas.
Let's first introduce the measure preservation of characteristic curves.
\begin{Lemma}\label{L3}
Suppose that $X(s;t,x,v),\, V(s;t,x,v)$ are  the characteristic curves,
the following transformation is  symplectic and preserves the measure, that is
\begin{align}
&(x,t)\rightarrow (X(s;t,x,v), V(s;t,x,v)), \nonumber\\
&{\rm d}X(s;t,x,v){\rm d}V(s;t,x,v)={\rm d}x{\rm d}v.\nonumber
\end{align}
\end{Lemma}
\begin{proof}
The proof of this lemma  is based on the characteristic equation being a Hamiltonian system and we omit the details here (cf. \cite{HV}, Lemma 7).
\end{proof}

Next, we provide the uniform estimates for $f^{n}, E^{n}$.

\begin{Lemma}\label{L4}
Under the assumptions of Theorem \ref{T1}, and assumping that there exists $n_{0}$ such that  $Q^{n}(t) \leq M$  for $n\geq n_{0}, 0\leq t\leq T$, we obtain that, 
for  $n\geq n_{0}+1, 0\leq t\leq T$, 
\begin{align}
&|E^{n}(t,x)|\leq C(T), \label{v5.18}\\
&|E^{n}(t,\cdot)|_{C^{\gamma}(\bar{\Omega})}\leq C(T),\,\, for\,\, any\,\, 0<\gamma <1, \label{v5.19}
\end{align}
where $C(T)$ depends only on $M,T,\|f_{0}\|_{L^{\infty}(\Omega\times\mathbb{R}^{2})}$.
\end{Lemma}
\begin{proof}
On the one hand, by the definition of $\rho^{n}$, we have
\begin{align}
|\rho^{n}(t,x)|\leq \|f^{n}\|_{L^{\infty}}\big(Q^{n}(t)\big)^{2}\leq M^{2}\|f_{0}\|_{L^{\infty}}. \nonumber
\end{align}
From the representation formula for the solutions of the Poisson equation, we can see that \eqref{v5.18} holds.

On the other hand, since $\dv E^{n}=\rho^{n}, \,\,\cl E^{n}=0,\,\, E^{n} \cdot n =h(x)$, we have the following estimate, for any $1< p< \infty$,
\begin{align}
\|E^{n}\|_{W^{1.p}}&\leq C\Big(\|\dv E^{n}\|_{L^{p}}+\|\cl E^{n}\|_{L^{p}}+\|E^{n}\cdot n\|_{L^{p}}+\|E^{n}\|_{L^{p}}\Big)\nonumber\\
& \leq C\Big(\|\rho^{n}\|_{L^{p}}+\|h(x)\|_{L^{p}}+\|E^{n}\|_{L^{p}}\Big)\nonumber\\
&\leq  C\Big(\|\rho^{n}\|_{L^{\infty}}+\|h(x)\|_{L^{p}}+\|E^{n}\|_{L^{\infty}}\Big)\nonumber\\
&\leq C(T),\nonumber
\end{align}
and then, according to the Sobolev embedding inequality, we conclude that $|E^{n}(t,\cdot)|_{C^{\gamma}(\bar{\Omega})}$ is uniformly bounded, that is, \eqref{v5.19} holds.
\end{proof}

Now we establish the estimate on $\|f^{n}\|_{C^{1;1,\lambda}_{t;(x,v)}}([0,T]\times \bar{\Omega}\times\mathbb{R}^{2})$  by the following lemma.

\begin{Lemma}
Let $Q^{n}(t) \leq M$  for $n\geq n_{0}, 0\leq t\leq T$. Then $f^{n}$ has the following  uniform bounds, for $n\geq n_{0}+1,$
\begin{align}
\|f^{n}\|_{C^{1;1,\lambda}_{t;(x,v)}([0,T]\times \bar{\Omega}\times\mathbb{R}^{2})}\leq C(T), \label{v5.20}
\end{align}
where $C(T)$ depends only on $M,\,T$.
\end{Lemma}
\begin{proof}
Firstly,  
by the estimates \eqref{v5.18}-\eqref{v5.19} and Lemma \ref{L2}, choosing $\gamma> \frac{1}{2}$, it follows that
\begin{align}
&\|E^{n}\|_{L^{\infty}([0,T],C^{\frac{1}{2}}(\Omega))}\leq C(T),\nonumber\\
&C_{1}^{n}(T)\alpha(0)\leq \alpha^{n}(t)\leq C_{2}^{n}(T)\alpha(0),\nonumber
\end{align}
where $C_{1}^{n}(T),C_{2}^{n}(T)$ depends only on $T,\|E^{n}\|_{L^{\infty}([0,T],C^{\frac{1}{2}}(\Omega))}$.
Thus, we can obtain
\begin{align}
C_{1}(T)\alpha(0)\leq \alpha^{n}(t)\leq C_{2}(T)\alpha(0). \label{v5.21}
\end{align}
 The characteristic equations can been rewritten as,
\begin{align}
\frac{d l}{d t}=&\frac{1}{\sqrt{1+|v|^{2}}}\frac{\omega}{1-kx_{\bot}},\nonumber\\
\frac{d \omega}{d t}=&\sigma=E_{l} +\frac{1}{\sqrt{1+|v|^{2}}}\frac{k v_{\bot}\omega}{1-kx_{\bot}},\nonumber\\
\frac{d \alpha}{d t}=&v_{\bot}\Big(F(t,x_{\bot})-F(t,0)\Big)
  +\frac{v_{\bot}x_{\bot}h(x)}{\sqrt{1+|v|^{2}}}F(t,x_{\bot})\nonumber\\
  &-x_{\bot}\Big(\frac{1}{\sqrt{1+|v|^{2}}}\frac{\omega}{1-kx_{\bot}}\frac{\partial L(t,0)}{\partial l}+\sigma \frac{\partial L(t,0)}{\partial \omega}\Big),\nonumber\\
\frac{d \beta }{d t}=&-\frac{\pi}{\sqrt{2\alpha}}F(t,0)+\pi\frac{2x_{\bot}L(t,0)}{(2\alpha)^{3/2}}\Big(F(t,x_{\bot})-F(t,0)\Big)-\pi \frac{x_{\bot}v_{\bot}}{(2\alpha)^{\frac{3}{2}}}
 \Big(\sigma \frac{\partial L(t,0)}{\partial\omega}\nonumber\\
   &+\frac{1}{\sqrt{1+|v|^{2}}}\frac{\omega}{1-k x_{\bot}}
  \frac{\partial L(t,0)}{\partial l}+F(t,x_{\bot})\frac{v_{\bot}h(x)}{\sqrt{1+|v|^{2}}}\Big).\nonumber
\end{align}
Since the function $f^{n}$  propagates along the  characteristic curves,
the key is to estimate the corresponding norm of the characteristic curves with respect to the initial data to obtain the    H\"older estimate of $ f^{n}$.
Let's introduce the notation as follows,
\begin{align}
&\xi=(l,\omega,\alpha,\beta),\quad \xi_{0}=(l_{0},\omega_{0},\alpha_{0},\beta_{0}),\nonumber\\
&[g]_{\lambda;\xi_{0}}=\sup_{|(x_{0},v_{0})-(x'_{0}-v'_{0})|\leq 1}\frac{|g(x_{0},v_{0})-g(x'_{0},v'_{0})|}{|\xi_{0}-\xi'_{0}|^{\lambda}}.\nonumber
\end{align}
We have the following claim, for any $\lambda< \mu$,
\begin{align}\label{claimxi}
\Big|\frac{d}{d t}\big([\xi]\big)_{\lambda,\xi_{0}}\Big|\leq C[\xi]_{\lambda,\xi_{0}}. 
\end{align}
In fact,
\begin{align}
\frac{d \alpha}{d t}=&v_{\bot}\Big(F(t,x_{\bot})-F(t,0)\Big)
  +\frac{v_{\bot}x_{\bot}h(x)}{\sqrt{1+|v|^{2}}}F(t,x_{\bot})\nonumber\\
  &-x_{\bot}\big(\frac{1}{\sqrt{1+|v|^{2}}}\frac{\omega}{1-kx_{\bot}}\frac{\partial L(t,0)}{\partial l}+\sigma \frac{\partial L(t,0)}{\partial \omega}\big),\nonumber\\
\frac{d \alpha'}{d t}=&v'_{\bot}\Big(F'(t,x'_{\bot})-F'(t,0)\Big)
  +\frac{v'_{\bot}x'_{\bot}h(x')}{\sqrt{1+|v'|^{2}}}F'(t,x'_{\bot})\nonumber\\
  &-x'_{\bot}\Big(\frac{1}{\sqrt{1+|v'|^{2}}}\frac{\omega'}{1-kx'_{\bot}}\frac{\partial L'(t,0)}{\partial l}+\sigma \frac{\partial L'(t,0)}{\partial \omega'}\Big).\nonumber
\end{align}
Subtracting the above two equations yields
\begin{align}
\Big|\frac{d (\alpha-\alpha')}{d t}\Big|\leq C\Big(|l-l'|+|\omega-\omega'|+|\alpha-\alpha'|+|\beta-\beta'|\Big).\label{v5.22}
\end{align}
Here, we have used \eqref{v5.21} and the following facts,
\begin{align}
|x_{\bot}-x'_{\bot}|=&\Big|\frac{\alpha}{L}\big(2+4H-\frac{\beta+\beta'}{\pi}\big)\cdot \frac{\beta-\beta'}{\pi}+\big(\frac{\alpha}{L}-\frac{\alpha'}{L'}\big)
\big(1-(1-\frac{\beta'-2\pi H}{\pi})^{2}\big)\Big|\nonumber\\
\leq &C\Big(|\alpha-\alpha'|+|\beta-\beta'|+|l-l'|+|\omega-\omega'|\Big),\nonumber\\
|L'-L|=&\Big|\sqrt{1+|v|^{2}}h(x)-\sqrt{1+|v'|^{2}}h(x')+k\omega^{2}-k\omega'^{2}\Big|\nonumber\\
\leq & C\Big(|x-x'|+|v-v'|+|\omega-\omega'|\Big)\nonumber\\
\leq & C\Big(|\alpha-\alpha'|+|\beta-\beta'|+|l-l'|+|\omega-\omega'|\Big),\nonumber\\
|F-F'|=&\Big|E_{\bot}(t,x)-E_{\bot}(t,x')-\big(\frac{1}{\sqrt{1+|v|^{2}}}-\frac{1}{\sqrt{1+|v'|^{2}}}\big)\frac{k \omega^{2}}{1-k x_{\bot}}\nonumber\\
&-\frac{1}{\sqrt{1+|v'|^{2}}}\big(\frac{k \omega^{2}}{1-k x_{\bot}}-\frac{k \omega'^{2}}{1-k x'_{\bot}}\big)\Big|\nonumber\\
\leq & C\Big(|x-x'|+|v-v'|+|\omega-\omega'|+|x_{\bot}-x'_{\bot}|\Big)\nonumber\\
\leq & C\Big(|\alpha-\alpha'|+|\beta-\beta'|+|l-l'|+|\omega-\omega'|\Big).\nonumber
\end{align}
Similarly,
\begin{align}
\Big|\frac{d (l-l')}{d t}\Big|&\leq C\Big(|l-l'|+|\omega-\omega'|+|\alpha-\alpha'|+|\beta-\beta'|\Big),\label{v5.23}\\
\Big|\frac{d (\omega-\omega')}{d t}\Big|&\leq C\Big(|l-l'|+|\omega-\omega'|+|\alpha-\alpha'|+|\beta-\beta'|\Big),\label{v5.24}\\
\Big|\frac{d (\beta-\beta')}{d t}\Big|&\leq C\Big(|l-l'|+|\omega-\omega'|+|\alpha-\alpha'|+|\beta-\beta'|\Big).\label{v5.25}
\end{align}
By \eqref{v5.22}-\eqref{v5.25}, it follows that the above claim \ref{claimxi} holds. According to the claim, we have
\begin{align}
&C_{1}(T)[\xi(0)]_{\lambda,\xi_{0}}\leq [\xi]_{\lambda,\xi_{0}}\leq C_{2}(T)[\xi(0)]_{\lambda,\xi_{0}},\nonumber\\
&[\xi(0)]_{\lambda,\xi_{0}}=|\xi_{0}-\xi'_{0}|^{1-\lambda}.\nonumber
\end{align}
Therefore,  the following estimate holds
\begin{align}
[\xi]_{\lambda,\xi_{0}}\leq C(T), \label{v5.26}
\end{align}
where $C(T)$ depends only on $M,T.$ From this, we infer that
\begin{align}
\|f^{n}(t,\cdot,\cdot)\|_{C^{\lambda}_{x,v}(\bar{\Omega}\times\mathbb{R}^{2})}\leq C(T). \label{v5.27}
\end{align}
In fact, since
\begin{align}
\frac{|\xi(t)-\xi'(t)|}{|\xi_{0}-\xi'_{0}|}&=\frac{|\xi(t)-\xi'(t)|}{|\xi_{0}-\xi'_{0}|^{\lambda}}\cdot\frac{1}{|\xi_{0}-\xi'_{0}|^{1-\lambda}}
=[\xi]_{\lambda,\xi_{0}}\cdot\frac{1}{|\xi_{0}-\xi'_{0}|^{1-\lambda}}\nonumber\\
&\geq C_{1}(T)[\xi(0)]_{\lambda,\xi_{0}}\frac{1}{|\xi_{0}-\xi'_{0}|^{1-\lambda}}=C_{1}(T).\nonumber
\end{align}
That is, if $|\xi(t)-\xi'(t)|\leq 1$, we get $|\xi_{0}-\xi'_{0}|\leq \frac{1}{C_{1}(T)}$.
As a result,
\begin{align}
\sup_{|(x,v)-(x',v')|\leq 1}&\frac{|f^{n}(t,x,v)-f^{n}(t,x',v')|}{|x-x'|^{\lambda}+|v-v'|^{\lambda}}
=\sup_{|\xi(t)-\xi'(t)|\leq 1}\frac{|f_{0}(x_{0},v_{0})-f_{0}(x'_{0},v'_{0})|}{|\xi(t)-\xi'(t)|^{\lambda}}\nonumber\\
&=\sup_{|\xi_{0}-\xi'_{0}|\leq \frac{1}{C_{1}(T)}}\frac{|f_{0}(x_{0},v_{0})-f_{0}(x'_{0},v'_{0})|}{|\xi_{0}-\xi'_{0}|^{\lambda}}\cdot
\frac{1}{\Big(\frac{|\xi(t)-\xi'(t)|}{|\xi_{0}-\xi'_{0}|}\Big)^{\lambda}}\nonumber\\
&\leq\sup_{C_{1}|\xi_{0}-\xi'_{0}|\leq 1}\frac{|f_{0}(x_{0},v_{0})-f_{0}(x'_{0},v'_{0})|}{\big(C_{1}|\xi_{0}-\xi'_{0}|\big)^{\lambda}}\nonumber\\
&\leq C(T).\nonumber
\end{align}
Hence, the inequality \eqref{v5.27} holds.
From this,  we further infer that
\begin{align}
\|E^{n}(t,\cdot)\|_{C^{1,\lambda}}(\bar{\Omega})\leq C(T).\label{v5.28}
\end{align}
In fact, we first the estimate on $\|\rho^{n}(t,\cdot)\|_{C^{0,\lambda}}$,
\begin{align}
&|\rho^{n}(t,x)|\leq \|f^{n}\|_{L^{\infty}}\big(Q^{n}(t)\big)^{2}\leq M^{2}\|f_{0}\|_{L^{\infty}},\nonumber\\
&\frac{|\rho^{n}(t,x)-\rho^{n}(t,x')|}{|x-x'|^{\lambda}}\leq \|f^{n}(t,\cdot,\cdot)\|_{C^{\lambda}_{x,v}}\big(Q^{n}(t)\big)^{2}\leq C(T).\nonumber
\end{align}
Using again the Schauder estimate for the Poisson equation and \eqref{v5.6}, one has
\begin{align}
\|E^{n}(t,\cdot)\|_{C^{1,\lambda}(\bar{\Omega})}\leq C\big(\|\rho^{n}\|_{C^{0,\lambda}}+\|h\|_{C^{1,\lambda}}\big) \leq C(T).\nonumber
\end{align}
Similar to the estimate \eqref{v5.26}, we can obtain
\begin{align}
[\frac{\partial \xi}{\partial\xi_{0}}]\leq C(T).\label{v5.29}
\end{align}
 In fact, for simplicity, we introduce the following notation, $\xi=(l,\omega,\alpha,\beta)=\xi(t,\xi_{0})$,
\begin{align}
&\frac{d \xi}{d t}=:A(\xi),\nonumber\\
&\xi(0)=\xi_{0}.\nonumber
\end{align}
We denote $W=D_{\xi_{0}} \xi (t,\xi_{0}) $ as an $4\times4$ matrix, $I$ as identity matrix, then we have
\begin{align}
&\frac{d W}{d t}=D_{\xi}A(\xi) W,\nonumber \\
&W(0)=I.\nonumber
\end{align}
Therefore,  we have two conclusions about W, 
\begin{align}
C_{1}(T)|W(0)|\leq |W(t)|\leq C_{2}(T)|W(0)|,\label{v5.30}
\end{align}
and  
\begin{align}
\Big|\frac{d}{ d t}[W]_{\lambda,\xi_{0}}\Big|\leq C[W]_{\lambda,\xi_{0}}+C[\xi]_{\lambda,\xi_{0}}.\nonumber
\end{align}
Applying Gronwall's inequality, it follows that \eqref{v5.29} holds.
Finally, based on the inequlities \eqref{v5.29},\eqref{v5.30}, we can conclude that
$$\|f^{n}(t,\cdot,\cdot)\|_{C^{1,\lambda}_{(x,v)}}\leq C(T).$$
Using the equation \eqref{v3.2}, we have
$$\|f^{n}_{t}\|_{C^{\lambda}_{(x,v)}}\leq C(T).$$
Thus we deduce that \eqref{v5.20} is true.
\end{proof}

Next, we will prove the convergence of the iterative sequence $f^{n}$.
\begin{Proposition}\label{P5}
Under the assumptions of Theorem \ref{T1}, let us assume that $Q^{n}(t)\leq M$ for $n\geq n_{0}, t\in [0,T]$. Then there exists $f\in C^{1;1,\lambda}_{t;(x,v)}([0,T]\times \bar{\Omega}\times \mathbb{R}^{2})$ which satisfies the system \eqref{v1.5}-\eqref{v1.9} such that $f^{n}\,\rightarrow \,f\in C^{\nu;1,\lambda}_{t;(x,v)}([0,T]\times \bar{\Omega}\times \mathbb{R}^{2})$ when $n\rightarrow \infty, $ with $0<\lambda<\mu, 0<\nu<1.$
\end{Proposition}

\begin{proof}
First of all, we show that the iterative sequence $f^{n}$ is a Cauchy sequence in $L^{1}([0,T]\times\Omega\times\mathbb{R}^{2})$.
According to the iterative equations \eqref{v3.2}, it follows that $f^{n+1}-f^{n}$ satisfies the following equation,
\begin{align}
\big(f^{n+1}-f^{n}\big)_{t}+\frac{v}{\sqrt{1+|v|^{2}}}\cdot &\nabla_{x}\big(f^{n+1}-f^{n}\big)+\nabla_{x}\varphi^{n}\cdot\nabla_{v}\big(f^{n+1}-f^{n}\big)\nonumber\\
&=\nabla_{x}\cdot \big(\varphi^{n-1}-\varphi^{n}\big)\cdot \nabla_{v}f^{n}.\label{v5.31}
\end{align}
The characteristic  equations corresponding to the iterative equation \eqref{v3.2} are,
\begin{align}
&\frac{d X}{d s}=\frac{V(s)}{\sqrt{1+|V(s)|^{2}}},\quad \frac{d V(s)}{d s}=\nabla_{x}\varphi^{n}(s,X(s)),\nonumber\\
&X(t)=x,\quad v(t)=v.\nonumber
\end{align}
Integrating the equation \eqref{v5.31} along characteristic curves $(X(s),V(s))$, and using the fact $f^{n+1}(0,X(0),V(0))=f^{n}(0,X(0),V(0))$,  we obtain,
\begin{align}
\big(f^{n+1}-f^{n}\big)(t,x,v)=\int^{t}_{0}\nabla_{x} \big(\varphi^{n-1}-\varphi^{n}\big)(s,X(s)) \nabla_{v}f^{n}(s,X(s),V(s)){\rm d} s. \label{v5.32}
\end{align}
From the representation formula \eqref{v5.4} and the estimates \eqref{v5.5}, we have the following estimate on $\nabla_{x}\big(\varphi^{n-1}-\varphi^{n}\big)(s,x)$,
\begin{align}
\Big|\nabla_{x}\big(\varphi^{n-1}-\varphi^{n}\big)(s,x)\Big|\leq C\int_{\Omega}\frac{|\rho^{n}(y)-\rho^{n-1}(y)|}{|x-y|}{\rm d} y.\nonumber
\end{align}
Integrating \eqref{v5.32} with respect to the  variables $(x,v)$ over the space $\Omega\times\mathbb{R}^{2}$, we get,
\begin{align}
&\|f^{n+1}(t)-f^{n}(t)\|_{L^{1}(\Omega\times \mathbb{R}^{2})}\nonumber\\
\quad &\leq \int^{t}_{0}\int\int_{\Omega\times \mathbb{R}^{2}}\big|\big(\varphi^{n-1}-\varphi^{n}\big)(s,X(s))\big|\big|\nabla_{v}f^{n}(s,X(s),V(s))\big|
{\rm d} X(s){\rm d} V(s){\rm d} s\nonumber\\
\quad &\leq C  \int^{t}_{0}\int_{\Omega}G^{n}(y,s)\big|\rho^{n}(s,y)-\rho^{n-1}(s,y)\big|{\rm d} y{\rm d} s,\nonumber
\end{align}
where $G^{n}(y,s)$ is defined as,
\begin{align}
G^{n}(y,s)&=\int\int_{\Omega\times \mathbb{R}^{2}}\frac{1}{|X(s)-y|}|\nabla_{v}f^{n}(s,X(s),V(s))|{\rm d} X(s){\rm d} V(s)\nonumber\\
&=\int\int_{\Omega\times \mathbb{R}^{2}}\frac{1}{|x-y|}|\nabla_{v} f^{n}(s,x,v)|{\rm d} x{\rm d} v \nonumber\\
&=\int_{|y-x|\leq r}\frac{1}{|x-y|}\|\nabla_{v} f^{n}(s,x,\cdot)\|_{L^{1}_{v}}{\rm d} x\nonumber\\
&\quad+\int_{|y-x|\geq r}\frac{1}{|x-y|}\|\nabla_{v} f^{n}(s,x,\cdot)\|_{L^{1}_{v}}{\rm d} x \nonumber\\
&\leq C r \|\nabla _{v}f^{n}(t,x)\|_{L^{\infty}_{x}(L^{1}_{v})}+\frac{1}{r}\|\nabla _{v}f^{n}(t,\cdot,\cdot)\|_{L^{1}_{(x,v)}}\nonumber\\
&\leq C\|\nabla _{v}f^{n}(t,x)\|^{\frac{1}{2}}_{L^{\infty}_{x}(L^{1}_{v})}\|\nabla _{v}f^{n}(t,\cdot,\cdot)\|^{\frac{1}{2}}_{L^{1}_{(x,v)}}.\nonumber
\end{align}
Here, we used Lemma \ref{L3} and chose $r=\|\nabla _{v}f^{n}(t,x)\|^{-\frac{1}{2}}_{L^{\infty}_{x}(L^{1}_{v})}\|\nabla _{v}f^{n}(t,\cdot,\cdot)\|^{\frac{1}{2}}_{L^{1}_{(x,v)}}$.

Since $\|f^{n}\|_{C^{1;1,\lambda}_{t,(x,v)}}\leq C(T)$, it follows that $\|f^{n}\|_{W^{1,\infty}}\leq C(T)$. Furthermore, we  have 
\begin{align}
\|\nabla _{v}f^{n}(t,x)\|_{L^{\infty}_{x}(L^{1}_{v})}&\leq C(T)\|f^{n}\|_{W^{1,\infty}}\big(Q^{n}(t)\big)^{2}\leq C(T),\nonumber\\
\|\nabla _{v}f^{n}(t,\cdot,\cdot)\|_{L^{1}_{(x,v)}}&\leq C(T)\|f^{n}\|_{W^{1,\infty}}\big(Q^{n}(t)\big)^{2}|\Omega|\leq C(T),\nonumber
\end{align}
where $|\Omega|$ represents  the measure of $\Omega$, then we infers that $G^{n}(y,s)\leq C(T)$.
Therefore, we obtain the following recursive inequality,
\begin{align}
&\|f^{n+1}(t)-f^{n}(t)\|_{L^{1}(\Omega\times \mathbb{R}^{2})}\nonumber\\
&\quad \leq C(T)\int^{t}_{0}\|f^{n+1}(s)-f^{n}(s)\|_{L^{1}(\Omega\times \mathbb{R}^{2})}{\rm d}s, \label{v5.33}
\end{align}
where $C(T)$ depends only on $T,M$ and the initial data.

For convenience, we denote $A_{n+1}(t):=\|f^{n+1}(t)-f^{n}(t)\|_{L^{1}(\Omega\times \mathbb{R}^{2})}$, then
we claim that there exists $0<\kappa<1, \epsilon_{0}$ (depending only on T), such that if $0< t< \epsilon_{0}$,
$$A_{n+1}(t)\leq C_{1}\kappa^{n}.$$
Here, we provide some simple deductions, when $0< t< \epsilon_{0}$, by \eqref{v5.33},
\begin{align}
A_{n}(t)\leq C(T)\int^{\epsilon_{0}}_{0}A_{n-1}(s){\rm d}s\leq\cdots \leq \|A_{1}\|_{L^{\infty}}\big(C(T)\epsilon_{0}\big)^{n-1}, \nonumber
\end{align}
choosing $C(T)\epsilon_{0}=\kappa<1$,  and it follows that the claim holds.

If $\epsilon_{0}< t\leq 2\epsilon_{0}$, using again \eqref{v5.33}, we have,
$$A_{n+1}(t)\leq C_{1}\kappa^{n}+C(T)\int^{t}_{\epsilon_{0}}A_{n-1}(s){\rm d}s.$$
From this iterative inequality, we can obtain,
$$A_{n+1}(t)\leq C_{1}(n+1)\kappa^{n}.$$
In a similar way, we can obtain if $ k\epsilon_{0}< t\leq (k+1)\epsilon_{0}, k=0,1,2,\ldots$,
$$A_{n+1}(t)\leq C_{1}(n+1)^{k}\kappa^{n}.$$
From the iterative discussion above, we can determine that $f^{n}$ is a Cauchy sequence in $L^{\infty}\big([0,T],L^{1}(\Omega\times\mathbb{R}^{2})\big)$.

Now based on   \eqref{v5.20} and  the fact that $f^{n}$ is a Cauchy sequence in $L^{1}([0,T]\times\Omega\times\mathbb{R}^{2})$,   using the interpolation discussion, we can prove that
for any $0< \lambda<\mu, 0<\nu < 1$, $f^{n}$ is also Cauchy in $C^{\nu;1,\lambda}_{t;(x,v)}([0,T]\times \bar{\Omega}\times \mathbb{R}^{2})$.

In fact, according to \eqref{v5.20}, applying the following  interpolation formula, for any $p>1$,
\begin{align}
&\|f^{n}\|_{W^{1,\infty}([0,T]\times\bar{\Omega}\times\mathbb{R}^{2})}\leq C(T),\nonumber\\
&\|f^{n}\|_{L^{p}([0,T]\times\bar{\Omega}\times\mathbb{R}^{2})}\leq C\|f^{n}\|^{\frac{1}{P}}_{L^{\infty}([0,T]\times\bar{\Omega}\times\mathbb{R}^{2})}\|f^{n}\|^{1-\frac{1}{P}}_{L^{1}([0,T]\times\bar{\Omega}\times\mathbb{R}^{2})},\nonumber\\
&\|\nabla f^{n}\|_{L^{p}([0,T]\times\bar{\Omega}\times\mathbb{R}^{2})}\leq
C\|f^{n}\|^{\vartheta}_{W^{1,\infty}([0,T]\times\bar{\Omega}\times\mathbb{R}^{2})}\|f^{n}\|^{1-\vartheta}_{L^{1}([0,T]\times\bar{\Omega}\times\mathbb{R}^{2})},\nonumber
\end{align}
where $\vartheta=1-\frac{5}{6p}$. Thus, $f^{n}$ is  a Cauchy sequence in $W^{1,p}([0,T]\times\Omega\times\mathbb{R}^{2})$.
Furthermore, based on Sobolev's embedding inequality, it follows that $f^{n}$ is  a Cauchy sequence in $C^{\tilde{\nu}}([0,T]\times\Omega\times\mathbb{R}^{2})$
 for any $0<\tilde{\nu}<1$.

According to the interpolation theorem in Schauder spaces (cf. Chapter 6 in \cite{GT}) and interpolating between $ C^{1,\tilde{\lambda}}_{(x,v)}$ and $C^{\tilde{\nu}}$, we have
$f^{n}$ is Cauchy sequence in $C^{1,\lambda}_{(x,v)}([0,T]\times\Omega\times\mathbb{R}^{2})$ for $0<\lambda<\tilde{\lambda}$. Then  in the same way  interpolating in $C^{1}_{t}([0,T]\times\Omega\times\mathbb{R}^{2})$ and $C^{0}([0,T]\times\Omega\times\mathbb{R}^{2})$, we see that $f^{n}$ is also Cauchy sequence in $C^{\nu}_{t}([0,T]\times\Omega\times\mathbb{R}^{2})$ for $0<\nu<1$. 
Therefore, $f^{n}$ is also Cauchy in $C^{\nu;1,\lambda}_{t;(x,v)}([0,T]\times \bar{\Omega}\times \mathbb{R}^{2})$.

Finally, we explain  why $f \in C^{1;1,\lambda}_{t;(x,v)}([0,T]\times \bar{\Omega}\times \mathbb{R}^{2})$. Integrating the equation \eqref{v3.2} over $[0,T]$ yields
$$ f^{n}(t)=f_{0}-\int^{t}_{0}\Big(\frac{v}{\sqrt{1+|v|^{2}}}\cdot\nabla_{x}f^{n}(s)+\nabla_{x}\varphi^{n-1}(s)\cdot\nabla_{v}f^{n}(s)\Big) {\rm d}s.$$
Taking the limit in the above equation when $n\rightarrow \infty$, we have
$$ f(t)=f_{0}-\int^{t}_{0}\Big(\frac{v}{\sqrt{1+|v|^{2}}}\cdot\nabla_{x}f(s)+\nabla_{x}\varphi(s)\cdot\nabla_{v}f(s)\Big) {\rm d}s.$$
That is,  $f$ is continuously differentiable with respect to $t$. We complete the proof of the proposition.
\end{proof}

\subsection{Prolongability of uniform estimates for the functions $f^{n}$}
\begin{Proposition}\label{P6}
Let us assume that for some $T>0$ there exists $M>0$ and $n_{0}\geq 0$ such that $Q^{n}(t) \leq M$  for any $n\geq n_{0}$ and $t\in [0,T]$. Then, there exists $\tau >0$ depending only on $M,\|f_{0}\|_{L^{\infty}}$, such that, 
$$Q^{n}(t) \leq 2M,$$
for $0\leq t\leq T+\tau$ and $n\geq n_{0}$.
\end{Proposition}
\begin{proof}
With the help of the representation formula  \eqref{v5.4}, it follows that
\begin{align}
|\nabla \varphi^{n}|&\leq C \int_{|x-y|\leq r}\frac{1}{|x-y|}|\rho^{n}(y)|{\rm d}y +C \int_{|x-y|\geq r}\frac{1}{|x-y|}|\rho^{n}(y)|{\rm d}y + \|h\|_{C^{1,\mu}}\nonumber\\
&\leq C\|\rho^{n}\|_{L^{\infty}}r+C\frac{\|\rho^{n}\|_{L^{1}}}{r}+ \|h\|_{C^{1,\mu}}\nonumber\\
&\leq C\|\rho^{n}\|^{\frac{1}{2}}_{L^{\infty}}\|\rho^{n}\|^{\frac{1}{2}}_{L^{1}}+ \|h\|_{C^{1,\mu}}\nonumber\\
&\leq C Q^{n}(t) \|f_{0}\|^{\frac{1}{2}}_{L^{\infty}}\|f_{0}\|^{\frac{1}{2}}_{L^{1}}+ \|h\|_{C^{1,\mu}},\nonumber
\end{align}
where we used the following facts,
\begin{align}
&r:=\|f_{0}\|^{-\frac{1}{2}}_{L^{\infty}}\|f_{0}\|^{\frac{1}{2}}_{L^{1}},\quad \|\rho^{n}\|_{L^{1}}=\|f_{0}\|_{L^{1}},\nonumber\\
&|\rho^{n}(t,x)|=\big|\int_{\mathbb{R}^{2}}f^{n}(t,x,v){\rm d} v\big|\leq \|f_{0}\|_{L^{\infty}}\big(Q^{n}(t)\big)^{2}.\nonumber
\end{align}
According to the characteristic equation \eqref{v4.2}, we have  for $t\geq T$,
\begin{align}
v^{n+1}(t)-v^{n+1}(T)=\int_{T}^{t}\nabla \varphi^{n}(s,X(s)){\rm d}s,\nonumber
\end{align}
furthermore, by \eqref{v5.17}, it follows that
\begin{align}
Q^{n+1}(t)&\leq Q^{n+1}(T)+C\|f_{0}\|^{\frac{1}{2}}_{L^{\infty}}\|f_{0}\|^{\frac{1}{2}}_{L^{1}}\int^{t}_{T} Q^{n}(s){\rm d}s+C(t-T)\nonumber\\
&\leq M+C\|f_{0}\|^{\frac{1}{2}}_{L^{\infty}}\|f_{0}\|^{\frac{1}{2}}_{L^{1}}\int^{t}_{T} Q^{n}(s){\rm d}s +C(t-T),\nonumber
\end{align}
where $C$ is independent of $n,Q^{n}$.
Defining $P^{n}(t)=\max\big\{Q^{l}(t): n_{0}\leq l\leq n\big\}$, we have
\begin{align}
P^{n+1}(t)
&\leq M+C\|f_{0}\|^{\frac{1}{2}}_{L^{\infty}}\|f_{0}\|^{\frac{1}{2}}_{L^{1}}\int^{t}_{T} P^{n+1}(s){\rm d}s +C(t-T).\nonumber
\end{align}
  On the basis of a Gronwall-type   argument, choosing $\tau=\frac{M}{C\big(2M\|f_{0}\|^{\frac{1}{2}}_{L^{\infty}}\|f_{0}\|^{\frac{1}{2}}_{L^{1}}+1\big)}$  yields
\begin{align}
Q^{n}(t)\leq 2M, \quad n\geq n_{0},\,\, 0\leq t\leq T+\tau. \nonumber
\end{align}
\end{proof}

\subsection{The sequence $Q^{n}(t)$ converge to $Q(t)$}
\begin{Proposition}\label{P7}
Let us assume that $Q(t)\leq M,  Q^{n}(t)\leq M$ for $n\geq n_{0}, 0\leq t\leq T$. If $f^{n}\,\rightarrow \,f\in C^{\nu;1,\lambda}_{t;(x,v)}([0,T]\times \bar{\Omega}\times \mathbb{R}^{2})$ for any $0<\lambda<\mu, 0<\nu<1$,  then we have $Q^{n}(t)\rightarrow Q(t)$ uniformly on $[0,T]$.
\end{Proposition}
\begin{proof}
Due to the Lemma \ref{L2}, we know that the characteristic curves away from the singular set  in the initial state remain away from it
during their evolution. For these characteristic curves, we can estimate their difference as $n\rightarrow \infty$.
It is similar to the proof of the analogous result on Vlasov-Poisson system (cf. \cite{HV,HV1}).
In fact, since the sequence $f^{n}$ is uniformly bounded on $n$ in the space $C^{1;1,\lambda}_{t;(x,v)}([0,T]\times \bar{\Omega}\times \mathbb{R}^{2})$, similar to the proof of Theorem\ref{T2}, we see that  the number of bounces is uniformly bounded with respect to $n$ in the time interval $[0,T]$.
Moreover, the time series when the characteristic curves of $f^{n}$ bounce to the boundary converge to the time when the charateristics of $f$ collide with the boundary.
Since $E^{n}\rightarrow E$, it follows that the characteristic curves $\big(X^{n}(s;0,x_{0},v_{0}),V^{n}(s;0,x_{0},v_{0})\big)$ of $f^{n}$ converge to the ones $\big(X(s;0,x_{0},v_{0}),V(s;0,x_{0},v_{0})\big)$ of $f$  between bounces. Also we have the following estimate, 
\begin{align}
|X^{n}(S)-X(S)|=&v_{0}\int^{s}_{0}\big(\frac{1}{\sqrt{1+|V^{n}(\tau)|^{2}}}-\frac{1}{\sqrt{1+|V(\tau)|^{2}}}\big) {\rm d}\tau \nonumber\\
&+\int^{s}_{0}\int^{\tau}_{0}\frac{E^{n}(\mu,X^{n}(\mu))-E(\mu,X^{n}(\mu))}{\sqrt{1+|V^{n}(\tau)|^{2}}}{\rm d}\mu{\rm d}\tau \nonumber\\
&+\int^{s}_{0}\int^{\tau}_{0}E(\mu,X(\mu))\big(\frac{1}{\sqrt{1+|V^{n}(\tau)|^{2}}}-\frac{1}{\sqrt{1+|V(\tau)|^{2}}}\big)\nonumber\\
\leq &C(T)\big(\int^{s}_{0}|V^{n}(s)-V(s)|{\rm d}\tau +\int^{s}_{0}|X^{n}(s)-X(s)|{\rm d}\tau\big)\nonumber\\
&+C(T)\|E^{n}-E\|_{L^{\infty}}.\nonumber
\end{align}
Similarly,
\begin{align}
|V^{n}(S)-V(S)|\leq C(T)\int^{s}_{0}|X^{n}(s)-X(s)|{\rm d}\tau +C(T)\|E^{n}-E\|_{L^{\infty}}.\nonumber
\end{align}
Let us define $Z(s)=|X^{n}(S)-X(S)|+|V^{n}(S)-V(S)|$, then
$$Z(s)\leq C(T)\int^{s}_{0}Z(\tau){\rm d}\tau+C(T)\|E^{n}-E\|_{L^{\infty}}.$$
By Gronwall's inequality, it follows that
$$Z(s)\leq C(T)\|E^{n}-E\|_{L^{\infty}}\rightarrow 0.$$
Because the number of bounces is uniformly  bounded and $|V|$ remains unchanged before and after the bounce, we conclude that
the functions $|V^{n}|$ converge uniformly to $|V(s)|$ when $n\rightarrow \infty$. Meanwhile, in accordance with the definition of $Q^{n}(t)$,
we obtain that $Q^{n}(t)\rightarrow Q(t)$, and the proof is completed.
\end{proof}

\section{Global Bound for Q(t)}\label6

In this part, we aim to determine the upper bound  of the function $Q(t)$ in any given time interval $[0,T]$. By doing so, we can subsequently extend the solution of the system
 \eqref{v1.5}-\eqref{v1.9} to intervals of arbitrary length.
\begin{Proposition}\label{P8}
Assume that $f_{0}\in C^{1,\mu}(\Omega\times \mathbb{R}^{2}),\,\,0<\mu<1$, and  $f\in C^{1;1,\mu}_{t,(x,v)}([0,T]\times\bar{\Omega}\times \mathbb{R}^{2})$ solves
the system \eqref{v1.5}-\eqref{v1.9} with $\lambda \in (0,1), 0< T<\infty$. Then there exists $\zeta(T)< \infty$ depending only on $T, Q(0),\|f_{0}\|_{C^{1,\mu}(\Omega\times \mathbb{R}^{2})} $ such that
\begin{align}
Q(t)\leq \zeta(T), \quad 0\leq t\leq T. \label{v6.1}
\end{align}
\end{Proposition}
\begin{proof}
We consider two cases on $Q(t)$, one case is within the domain, and the other case is near the boundary of the domain.
  
 {\bf Case 1.} We consider the characteristic equations within the domain.
By the representation formula  \eqref{v5.4}, we have
\begin{align}
|E(t,x)|&\leq C \int_{|x-y|\leq r}\frac{1}{|x-y|}|\rho(y)|{\rm d}y +C \int_{|x-y|\geq r}\frac{1}{|x-y|}|\rho(y)|{\rm d}y + \|h\|_{C^{1,\mu}}\nonumber\\
&\leq C\|\rho\|_{L^{\infty}}r+C\frac{\|\rho\|_{L^{1}}}{r}+ \|h\|_{C^{1,\mu}}\nonumber\\
&\leq C Q(t) \|f_{0}\|^{\frac{1}{2}}_{L^{\infty}}\|f_{0}\|^{\frac{1}{2}}_{L^{1}}+ \|h\|_{C^{1,\mu}},\nonumber
\end{align}
then, on the basis of $\frac{d V}{ds}= E(s,X(s))$, it follows that
\begin{align}
|V(t)|\leq C+C\int^{t}_{0}Q(s){\rm d}s,\nonumber
\end{align}
where $C$ depends only on $T$ and the initial data.
  
 {\bf Case 2.}  We handle the characteristic equations near the boundary. By \eqref{v2.3},
\begin{align}
&\frac{d \omega}{d s}=E_{l}-\frac{1}{\sqrt{1+|V|^{2}}}\frac{k v_{\bot}\omega}{1-kx_{\bot}},\nonumber\\
&\frac{d v_{\bot}}{d s}=E_{\bot}-\frac{1}{\sqrt{1+|V|^{2}}}\frac{k \omega^{2}}{1-kx_{\bot}}.\nonumber
\end{align}
Then, we obtain
\begin{align}
\big|\frac{d \omega}{d s}\big|+\big|\frac{d v_{\bot}}{d s}\big|\leq \|E\|_{L^{\infty}}+ CQ(s).\nonumber
\end{align}
Thus, we also get
\begin{align}
|V(t)|\leq C+C\int^{t}_{0}Q(s){\rm d}s,\nonumber
\end{align}
where $C$ depends only on $T,\,\partial\Omega$,and the initial data.

Combining the above two cases as well as the definition of $Q(t)$ yields
\begin{align}
Q(t)\leq C+C\int^{t}_{0}Q(s){\rm d}s.\nonumber
\end{align}
Therefore, based on Gronwall's inequality, there exists $\zeta(T)$ such that
$$Q(t)\leq \zeta(T),\quad 0\leq t\leq T.$$
\end{proof}

Finally, we provide the proof of Theorem \ref{T1}.
\begin{proof}[Proof of Theorem \ref{T1}]
In order to obtain the global existence of the solution of the system \eqref{v1.5}-\eqref{v1.9},   we need to establish that the function sequence $f^{n}$ of the iterative equation \eqref{v3.1}-\eqref{v3.6} converges to the solution of \eqref{v1.5}-\eqref{v1.9} for any time $t$. If the functions $Q^{n}(t)$ are uniformly bounded in any compact set about $t$, then we can conclude the desired limit on the basis of Propositon \ref{P5}.

  For this reason, we will first set 
\begin{align}
J(t):=\sup_{n}Q^{n}(t),\quad \lim_{t\rightarrow T_{\max}}J(t)=\infty. \nonumber
\end{align}
It is easy to see that $J(t)$ is a monotonically increasing function with respect to $t$. It suffices to prove that $T_{\max}=\infty$, which can be derived by contradiction. If $T_{\max} <\infty$,
let $\zeta(T_{\max})$ be as in  Proposition \ref{P8}, then based on  Proposition \ref{P6}, we choose
$\tau=\tau(2\zeta(T_{\max}),\|f_{0}\|_{L^{\infty}})$. From the definition of $T_{\max} $, it can be seen that $Q^{n}(t)$ is uniformly bounded
in the time interval $[0,T_{\max}-\frac{\tau}{2} ]$. Therefore,  making use of Proposition \ref{P5} as well as Proposition \ref{P7} yields
\begin{align}
&f^{n}\rightarrow f \quad  in \,\, C^{\nu;1,\lambda}_{t,(x,v)},  \quad for\quad 0\leq t\leq T_{\max}-\frac{\tau}{2},\nonumber \\
&Q^{n}(t)\rightarrow Q(t) \quad for \quad 0\leq t\leq T_{\max}-\frac{\tau}{2}.\nonumber
\end{align}
Specially, let $\tilde{t}= T_{\max}-\frac{\tau}{2}$, we have
$$\lim_{n\rightarrow \infty}Q^{n}(\tilde{t})=Q(\tilde{t})\leq \zeta(T_{\max}).$$
So, if $n\geq n_{0}$ and $n_{0}$ is large enough, $Q^{n}(\tilde{t})\leq 2\zeta(T_{\max})$;  and   in the light of Proposition \eqref{v5.6}, we know that for
$0\leq t\leq T_{\max}+\frac{\tau}{2}$ and $n\geq n_{0}$,
$$Q^{n}(t)\leq 4\zeta(T_{\max}).$$
Consequently, it follows that $J(t)$ is bounded as $t\rightarrow T_{\max}.$ This contradicts the definition of $T_{\max}.$
So far, we have proven the global existence of the solution of the system \eqref{v1.5}-\eqref{v1.9} in $C^{1;1,\lambda}_{t,(x,v)}$ for some $0<\lambda <\mu.$

Finally, we  illustrate the uniqueness in Theorem \ref{T1}. Suppose that $f_{1},f_{2}$ are two solutions of \eqref{v1.5}-\eqref{v1.9} and satisfy the same initial and boundary conditions, 
similarly to the derivation of \eqref{v5.33}, we have
\begin{align}
\|f_{1}(t)-f_{2}(t)\|_{L^{1}}\leq C(T)\int^{t}_{0}\|f_{1}(s)-f_{2}(s)\|_{L^{1}}{\rm d}s. \nonumber
\end{align}
Naturally,  $f_{1}=f_{2}$ holds. We completes the proof of Theorem \ref{T1}.
\end{proof}

\section{Global Existence of Solutions with Dirichlet Boundary Conditions}

In the case of Dirichlet boundary conditions for the electric potential, we shall show the global existence of classical solutions for general convex smooth domains in this section. It is necessary to make precise estimates for the electric potential $ \varphi(t,x)$ and its first derivatives near the boundary of the domain because difficulties near the singular set in the Dirichlet problem could primarily be resolved by introducing a new method to obtain an analogous velocity lemma. Specifically, we have modified the local coordinate variables 
$$\alpha=\frac{v_{\bot}^{2}}{2}-\varphi(t,x)-L(t,l,0,\omega,v_{\bot})x_{\bot}$$ based on the previously mentioned Newman boundary conditions. We omit the proofs of the other claims in the earlier sections because they still hold in the Dirichlet boundary case. We then focus mainly on proving the velocity lemma.

Let a point $x\in\Omega$ be fixed, and $x_{0}$ represent the tangent point of $x$. To examine the evolution of the characteristic curves starting from  the singular set  $\{(x,v)\in \partial\Omega\times \mathbb{R}^{2}, v\cdot n_{x}=0\}$, it can be hypothesized that  $\tilde{x}$ is near the boundary $\partial\Omega$ and that $x_{0}\in\partial\Omega$ is the point nearest to it. By employing rotations and translations, one can set $x_{0}=(0,0), \tilde{x}=(\tilde{x}_{1},0)$ and $\Omega \subset \mathbb{R}^{2}_{+}\equiv \{(x_{1},x_{2})\in \mathbb{R}^{2}; x_{1}>0\}$,  and we can easily establish that the tangent line to $\partial\Omega$ at $x_{0}$  is represented by the expression $\partial\Omega$ at $x_{0}$ .

A concise elucidation of local coordinate variables $(\alpha,\beta)$ is provided as follows,
\begin{align}
\alpha(t,l,x_{\bot},\omega,v_{\bot})&=\frac{v_{\bot}^{2}}{2}-L(t,l,0,\omega,v_{\bot})x_{\bot}-\varphi(t,x)\nonumber\\
&=\frac{v_{\bot}^{2}}{2}+\big(-\sqrt{1+|v|^{2}}E_{\bot}(t,l,0)+k\omega^{2}\big)x_{\bot}-\varphi(t,l,x_{\bot}),\label{v7.1}\\
\beta(t, l,  x_{\bot}, \omega, v_{\bot})&=2\pi H(t,l,x_{\bot},\omega,v_{\bot})+\pi (1-\frac{v_{\bot}}{\sqrt{2\alpha}}).\label{v7.2}
\end{align}

\begin{Remark}
Based on the argument below, we can conclude that $\alpha$  is nonnegative and is equivalent to $x_{\bot}+v^{2}_{\bot}$. First, deriving $\varphi \leq 0, x\in\Omega$ is a straightforward process using the equation for $\varphi$ and the maximum principle. Second, $E_{\bot}=-\frac{\partial \varphi}{\partial n}< 0$ as given by Hopf Lemma, and $k\geq 0$ by virtue of the convexity of $\Omega$.
\end{Remark}

It is now essential that we prove some technical estimates of the Newtonian potential and its first derivatives in order to obtain the velocity lemma further.

Let $T>0$, from the equation \eqref{v1.5}, it can be concluded that $\rho(t,x)$ solves
\begin{align}
\partial_{t}\rho+\nabla\cdot j=0,\qquad  (t,x)\in[0,T]\times\Omega,\label{v7.3}
\end{align}
with $$j=\int\hat{v}f(t,x,v){\rm d}v \in \big(C^{1}([0,T]\times\Omega)\big)^{3}.$$

\begin{Lemma}\label{L3}
Assume further that  $\varphi(t,x)$ satisfies the following boundary value problem
\begin{align}
\Delta \varphi(t,x)&=\rho(t,x), \qquad (t,x)\in[0,T]\times\Omega,\nonumber\\
\varphi(t,x)&=0,\qquad \qquad x\in\partial\Omega.\nonumber
\end{align}
Then we have
\begin{align}
\Big|\frac{\partial \varphi}{\partial t}(t,\tilde{x})\Big|\leq C\tilde{x}_{1}\big(1+|\log\tilde{x}_{1}|+|\log\tilde{x}_{1}|^{2}\big),\nonumber
\end{align}
where $C>0$ depends only on $L=diam \Omega$ and $\|j\|_{\infty}$.
\end{Lemma}
\begin{proof}
Let's set $R=|\tilde{x}|\ll 1$, and do the following change of variables,
$$X=\frac{x}{R},\quad Y=\frac{y}{R},\quad \tilde{X}=\frac{\tilde{x}}{R}=(1,0).$$
The Green function for the domain $\Omega$ is denoted as $G$, therefore $\varphi$  can be expressed as
$$\varphi(t,x)=\int_{\Omega}\rho(t,y)G(x,y){\rm d}y,$$
then, letting $\Omega_{R}=\frac{\Omega}{R}$ and $H(X,Y):=G(RX,RY)$,   we obtain
\begin{align}
|\varphi_{t}(t,x)|&\leq\|j\|_{\infty}\int_{\Omega}|\nabla_{y}G(x,y)|{\rm d}y \nonumber\\
&=R\|j\|_{\infty}\int_{\Omega_{R}}|\nabla_{Y}H(X,Y)|{\rm d}Y.\nonumber
\end{align}
We therefore only need to prove that
\begin{align}
\int_{\Omega_{R}}|\nabla_{Y}H(X,Y)|{\rm d}Y\leq C\big(1+|\log R|+|\log R|^{2}\big).\label{v7.4}
\end{align}
The integration region must therefore be split into two parts, which are $|Y|\leq 4$ and $|Y|\geq 4$.

{\bf Case 1.}  When $|Y|\leq 4$, let $X=(X_{1},X_{2}),\,X^{*}=(-X_{1},X_{2}) $ is the reflextion of X with respect to the coordinate axis $\{X_{1}=0\}$.
First of all we define
$$\bar{H}(X,Y)=-\frac{1}{2\pi}\big(\log|Y-X|-\log|Y-X^{*}|\big),$$
which is the Green function for the half-space restricted to $\Omega_{R}\times\Omega_{R}$.

From a straightforward computation, it can be deduced that
\begin{align}
\int_{|Y|\leq 4}|\nabla_{Y}\bar{H}(\tilde{X},Y)|{\rm d}Y&\leq \frac{1}{2\pi}\int_{|Y|\leq 4}\big(\frac{1}{|Y-\tilde{X}|}+\frac{1}{|Y-\tilde{X}^{*}|}\big){\rm d}Y\nonumber\\
&\leq C. \nonumber
\end{align}

Let $W(X,Y)=H(X,Y)-\bar{H}(X,Y)$ and $\psi(X,Y)=\nabla_{Y}W(X,Y)$, then the function $\psi(X,Y)$ satisfies the following system:
\begin{align}
&\Delta_{X}\psi(X,Y)=0, \qquad X,Y\in\Omega_{R},\nonumber\\
&\psi(X,Y)=-\frac{1}{2\pi}\nabla_{Y}\big(\log|Y-X|-\log|Y-X^{*}|\big),\quad X\in\partial\Omega_{R}.\nonumber
\end{align}
If $\text{dist}(Y,\partial\Omega_{R})\geq 1$, then $|Y-X|\geq 1$ for all $X\in\partial\Omega_{R}$,  and we get
\begin{align}
|\psi(X,Y)|&=\frac{1}{2\pi}\big|\frac{Y-X}{|Y-X|^{2}}-\frac{Y-X^{*}}{|Y-X^{*}|^{2}}\big|\nonumber\\
&\leq C \big(\frac{1}{|Y-X|}+\frac{1}{|Y-X^{*}|}\big)\leq C\frac{1}{|Y-X|}\leq C.\nonumber
\end{align}
 According to the maximum principle, there is a constant $C>0$ uniformly with respect to $Y$, such that
 \begin{align}
|\psi(X,Y)|\leq C, \qquad  \forall X\in \Omega_{R}.\nonumber
\end{align}

Although the corresponding estimates  of $\psi(X,Y)$  become more intricate when the distance between $Y$ and $\partial\Omega_{R}$ is less than 1, the question can be resolved through the construction of a supersolution using the Poisson integral formula.
 Locating a boundary point $Y_{0}\in\partial\Omega_{R}$ such that $\text{dist}(Y,Y_{0})=\text{dist}(Y,\partial\Omega_{R})$ is a simple task. Then the following facts are true for $X\in \partial\Omega_{R}$.

If $|X-Y_{0}|\geq 2$, the triangle inequality  yields $|X-Y|\geq 1$ and then we get
\begin{align}
|\psi(X,Y)|\leq C\frac{1}{|X-Y|}\leq C. \nonumber
\end{align}

If $|X-Y_{0}|\leq R$, then we have
$$|\psi(X,Y)|\leq C\frac{1}{|X-Y|}.$$

If $R\leq|X-Y_{0}|\leq 2$, with the help of the Taylor theorem and the convexity of $\Omega$, we can deduce
$$|\psi(X,Y)|\leq \frac{C}{|X-Y_{0}|^{2}}.$$

To sum up,  we define
\begin{equation}\nonumber
\tilde{\psi}(X,Y)=\left\{\begin{aligned}
 &C, \qquad  if\,\, |X-Y_{0}|\geq 2, \\
&\frac{C}{|X-Y|}, \quad  if\,\, |X-Y_{0}|\geq R,\\
&\frac{C}{|X-Y_{0}|^{2}}, \quad  if\,\, R\leq|X-Y_{0}|\leq 2.\\
\end{aligned}  \right.
\end{equation}
Thus, for $|Y|\leq 4$ and $\text{dist}(Y,\partial\Omega_{R})\leq 1$, by means of the maximum principle, we have the following estimate, 
\begin{align}
|\psi(\tilde{X},Y)|\leq & C\int_{\partial\Omega_{R}}\frac{1}{|\tilde{X}-\xi|}|\tilde{\psi}(\xi,Y)|{\rm d}l
\leq C\int_{\partial\Omega_{R}}\frac{1}{1+|\xi|}|\tilde{\psi}(\xi,Y)|{\rm d}l\nonumber\\
=& C\int_{|\xi-Y_{0}|\leq R}\frac{1}{1+|\xi|}|\tilde{\psi}(\xi,Y)|{\rm d}l
+C\int_{R\leq|\xi-Y_{0}|\leq 2}\frac{1}{1+|\xi|}|\tilde{\psi}(\xi,Y)|{\rm d}l\nonumber\\
&+C\int_{|\xi-Y_{0}|\geq 2}\frac{1}{1+|\xi|}|\tilde{\psi}(\xi,Y)|{\rm d}l.\nonumber
\end{align}
Next, we estimate each item on the right end of the above equality,
\begin{align}
\int_{|\xi-Y_{0}|\geq 2}\frac{1}{1+|\xi|}|\tilde{\psi}(\xi,Y)|{\rm d}l&\leq C \int_{|\xi-Y_{0}|\geq 2}\frac{1}{1+|\xi|}{\rm d}l
\leq C |\log R|,\nonumber\\
\int_{R\leq|\xi-Y_{0}|\leq 2}\frac{1}{1+|\xi|}|\tilde{\psi}(\xi,Y)|{\rm d}l&\leq C \int_{R\leq|\xi-Y_{0}|\leq 2}\frac{1}{1+|\xi|}\cdot\frac{1}{|\xi-Y_{0}|^{2}}{\rm d}l
\leq C ,\nonumber
\end{align}
and leting $\eta=Y-Y_{0}$ and $|X-Y|=|(X-Y_{0})-\eta|$,   we have
\begin{align}
\int_{|\xi-Y_{0}|\leq R}\frac{1}{1+|\xi|}|\tilde{\psi}(\xi,Y)|{\rm d}l&\leq C \int_{|\xi-Y_{0}|\leq R}\frac{1}{1+|\xi|}\frac{1}{|\xi-Y|}{\rm d}l\nonumber\\
&=C\int_{|\xi-Y_{0}|\leq R}\frac{1}{1+|\xi|}\frac{1}{|\xi-Y_{0}-\eta|}{\rm d}l\nonumber\\
&\leq  C \int_{|\xi-Y_{0}|\leq R}\frac{1}{|\xi-Y_{0}|+|\eta|}{\rm d}l\nonumber\\
&\leq  C \int_{|\xi-Y_{0}|\leq R}\frac{1}{|\xi|+|\eta|}{\rm d}\xi
\leq C\log (1+\frac{R}{|\eta|})\nonumber\\
&\leq C+ C |\log R|+C|\log dist(Y,\partial\Omega_{R})|.\nonumber
\end{align}
Thus, together we obtain
\begin{align}
|\psi(\tilde{X},Y)|\leq C+ C |\log R|+C|\log dist(Y,\partial\Omega_{R})|,\nonumber
\end{align}
and hence,
\begin{align}
\int_{dist(Y,\partial\Omega_{R})\leq 1, |Y|\leq 4}|\psi(\tilde{X},Y)|{\rm d}Y&\leq C + C |\log R|\nonumber\\
&+C\int_{dist(Y,\partial\Omega_{R})\leq 1, |Y|\leq 4}|\log dist(Y,\partial\Omega_{R})|{\rm d}Y\nonumber\\
&\leq C+ C |\log R|.\nonumber
\end{align}
Therefore,  we deduce that
\begin{align}
\int_{ |Y|\leq 4}|\nabla_{Y}H(\tilde{X},Y)|{\rm d}Y&\leq \int_{ |Y|\leq 4}|\nabla_{Y}\bar{H}(\tilde{X},Y)|{\rm d}Y +\int_{ |Y|\leq 4}|\psi(\tilde{X},Y)|{\rm d}Y\nonumber\\
&\leq C+C|\log R|. \nonumber
\end{align}

{\bf Case 2.}  If $|Y|\geq 4$, we fix $Y=Y_{*}$, and rescale the  variables by $\eta=\frac{Y}{|Y_{*}|}, \zeta=\frac{X}{|Y_{*}|}$ such that 
$$\zeta,\eta\in \hat{\Omega}:=\frac{\bar{\Omega}_{R}}{|Y_{*}|}.$$
Define $g(\zeta,\eta)=H(|Y_{*}|\zeta,|Y_{*}|\eta)=H(X,Y)$. With the help of a change of variables, $g(\zeta,\eta)$ satisfies
$$\Delta_{\zeta}g(\zeta,\eta)=\frac{1}{|Y_{*}|}\delta(X-Y),$$
from the fact $\Delta_{X}H(X,Y)=\delta(X-Y)$.

Furthermore, let $\phi(\zeta,\eta)=|Y_{*}|\nabla_{\eta}g(\zeta,\eta)$, we obtain
\begin{align}
\Delta_{\zeta}\phi(\zeta,\eta)&=\nabla_{\eta}\delta(\zeta-\eta),\qquad   \zeta,\eta\in\hat{\Omega},\nonumber\\
\phi(\zeta,\eta)&=0,\qquad\qquad\qquad \zeta\in\partial\hat{\Omega}.\nonumber
\end{align}

Subsequently, we must further consider two cases depending on whether point $\eta$ is close to the boundary $\partial\hat{\Omega}$ or not.
\\
{\bf Case 2-1.} If $\text{dist}(\eta,\partial\hat{\Omega})\geq\frac{1}{10}$, $\psi(\zeta,\eta)$ is defined by
$$\psi(\zeta,\eta)=\phi(\zeta,\eta)+\frac{1}{2\pi}\nabla_{\eta}\log|\zeta-\eta|.$$
Then, the function $\psi(\zeta,\eta)$ satisfies the following system:
\begin{align}
\Delta_{\zeta}\psi(\zeta,\eta)&=0, \qquad \quad \zeta,\eta\in\hat{\Omega},\nonumber\\
\psi(\zeta,\eta)&=\frac{1}{2\pi}\nabla_{\eta}\log|\zeta-\eta|,\quad \zeta\in\partial\hat{\Omega}.\nonumber
\end{align}
The assumption yields $|\psi(\zeta,\eta)|\leq C$ for all $\zeta\in\partial\hat{\Omega}$. In the light of the maximum principle, we have
\begin{align}
|\psi(\zeta,\eta)|\leq C, \qquad \forall \zeta\in\partial\hat{\Omega}.\nonumber
\end{align}
Given that we are now considering the confined region as $\{(\zeta,\eta)|\zeta\in\hat{\Omega}, |\zeta|\leq\frac{1}{2},|\eta|=1\}$ and applying the boundary regularity theory to the Laplace operator, the following estimate is valid,
\begin{align}
|\nabla_{\zeta}\psi(\zeta,\eta)|\leq C,  \qquad \forall \zeta\in\hat{\Omega}.\nonumber
\end{align}
Thus, for the restricted region it can be concluded that $|\nabla_{\zeta}\phi(\zeta,\eta)|\leq C,$  since $\phi(0,\eta)=0$. By the Taylor theorem, we get
\begin{align}
|\phi(\tilde{\zeta},\eta)|\leq C \text{dist}(\tilde{\zeta},\partial\hat{\Omega})=\frac{C}{|Y_{*}|}.\nonumber
\end{align}
Therefore, by the above range of variable substitution, we finally obtain that
$$|\nabla_{Y}H(\tilde{X},Y_{*})|\leq\frac{C}{|Y_{*}|^{3}}.$$
{\bf Case 2-2.} If $\text{dist}(\eta,\partial\hat{\Omega})\leq\frac{1}{10}$, we fix $\eta=\eta_{0}$ and define $\bar{\eta}_{0}$ as the boundary point closest to $\eta_{0}$.
 Let $\zeta^{*}$ as the reflection point
with respect to the tangent line at $\bar{\eta}_{0}$, namely,
\begin{align}
\zeta^{*}=\zeta+2\text{dist}(\zeta,\{\eta\in \mathbb{R}^{2},(\eta-\bar{\eta}_{0})\cdot n(\bar{\eta}_{0})=0\})n(\bar{\eta}_{0}),\nonumber
\end{align}
where $n(\bar{\eta}_{0})$ is the outward normal vector. The functions $\bar{g}(\zeta,\eta),\omega(\zeta,\eta)$ are defined as follows,
\begin{align}
\bar{g}(\zeta,\eta)&=|Y_{*}|\nabla_{\eta}g(\zeta,\eta),\nonumber\\
\omega(\zeta,\eta)&=\bar{g}(\zeta,\eta)+\frac{1}{2\pi}\nabla_{\eta}\big(\log|\zeta-\eta|-\log|\zeta^{*}-\eta|\big).\nonumber
\end{align}
Then we have
\begin{align}
&\Delta_{\zeta}\omega(\zeta,\eta)=0, \qquad \qquad\qquad \zeta,\eta\in\hat{\Omega},\nonumber\\
&\omega(\zeta,\eta)=\frac{1}{2\pi}\big(\frac{\zeta-\eta}{|\zeta-\eta|^{2}}-\frac{\zeta^{*}-\eta}{|\zeta^{*}-\eta|^{2}}\big),\quad \zeta\in\partial\hat{\Omega}.\nonumber
\end{align}
Now, for $\zeta\in\partial\hat{\Omega}$, by the triangle inequality, the following estimate holds

\begin{equation}\nonumber
|\omega(\zeta,\eta_{0})|\leq\left\{\begin{aligned}
 &C, \qquad  if\,\, |\zeta-\bar{\eta}_{0}|\geq \frac{1}{8}, \\
&\frac{C}{|\zeta-\bar{\eta}_{0}|+|\eta_{0}-\bar{\eta}_{0}|}, \quad  if\,\, |\zeta-\bar{\eta}_{0}|\leq \frac{1}{8}.\\
\end{aligned}  \right.
\end{equation}

We denote $D=\text{dist}(\eta_{0},\partial\hat{\Omega})=|\eta_{0}-\bar{\eta}_{0}|$, and apply the Poisson kernel formula to obtain, for $\zeta\in \mathbf{B}_{\frac{1}{|Y_{*}|}}(0)\cap\hat{\Omega}$,
\begin{align}
|\omega(\zeta,\eta_{0})|&\leq\int_{\partial\hat{\Omega}}\frac{1}{|\zeta-\beta|}|\omega(\beta,\eta_{0})|{\rm d}l\nonumber\\
&\leq\int_{|\beta-\bar{\eta}_{0}|\leq\frac{1}{8}}\frac{1}{|\zeta-\beta|}\cdot\frac{C}{|\zeta-\bar{\eta}_{0}|+|\eta_{0}-\bar{\eta}_{0}|}{\rm d}l+
\int_{|\beta-\bar{\eta}_{0}|\geq\frac{1}{8}}\frac{C}{|\zeta-\beta|}{\rm d}l\nonumber\\
&=:M_{1}+M_{2}\nonumber.
\end{align}

First, the estimation of $M_{1}$ is performed  as follows, since
$$|\eta_{0}|=1,\,\, |\eta_{0}-\bar{\eta}_{0}|\leq\frac{1}{10},\,\, |\zeta|\leq \frac{1}{4},$$
we have
$$|\zeta-\bar{\eta}_{0}|\geq\frac{13}{20},$$
then, for $\beta\in \partial\hat{\Omega}$ and $|\beta-\bar{\eta}_{0}|\leq\frac{1}{8}$, the triangle inequality yields
$$|\zeta-\beta|\geq\frac{21}{40}.$$
Thus,
\begin{align}
M_{1}\leq C\int_{0}^{\frac{1}{8}}\frac{1}{\mu+D}{\rm d}\mu\leq C|\log D|.\nonumber
\end{align}

Second, we estimate  $M_{2}$.  For $\beta\in \partial\hat{\Omega}$ and $|\beta-\bar{\eta}_{0}|\geq\frac{1}{8}$, due to
$$|\tilde{\zeta}|=|\frac{|\tilde{X}|}{|Y_{*}|}|=\frac{1}{|Y_{*}|},\,\,\text{dist}(\tilde{\zeta},\partial\hat{\Omega})=\frac{1}{|Y_{*}|},\,\,
|\tilde{\zeta}-|\leq\frac{1}{10},\,\, |\tilde{\zeta}-\beta|\geq \frac{1}{|Y_{*}|} ,$$
we get,
$$|\tilde{\zeta}-\beta|\geq \frac{1}{3}\big(\frac{1}{|Y_{*}|}+|\beta|\big),$$
therefore,
\begin{align}
M_{2}&\leq C\int_{|\beta-\bar{\eta}_{0}|\geq\frac{1}{8}}\frac{|Y_{*}|}{1+|Y_{*}||\beta|}{\rm d}l\nonumber\\
&\leq C|Y_{*}|\int^{\frac{C}{|Y_{*}|R}}_{0}\frac{1}{1+|Y_{*}|\mu}{\rm d}\mu\nonumber\\
&\leq C |\log R|.\nonumber
\end{align}

Finally, we obtain
\begin{align}
|\omega(\tilde{\zeta},\eta_{0})|&\leq C\big(|\log R|+|\log D|\big),\nonumber\\
|\bar{g}(\tilde{\zeta},\eta_{0})|&\leq C\big(1+|\log R|+|\log D|\big),\nonumber
\end{align}
thus, by the above range of variable substitution, we have
\begin{align}
|\nabla_{Y}H(\tilde{X},Y_{*})|\leq C\big(\frac{1}{|Y_{*}|^{2}}+\frac{1}{|Y_{*}|^{2}}|\log R|+\frac{1}{|Y_{*}|^{2}}|\log {\rm dist}(\frac{Y_{*}}{|Y_{*}|}.\partial\hat{\Omega}) |\big).\nonumber
\end{align}
Therefore,  if $|Y|\geq 4$,
\begin{align}
\int_{4\leq |Y|\leq \frac{L}{R}}|\nabla_{Y}H(\tilde{X},Y)|{\rm d}^{2}Y
 \leq C & \int_{4\leq |Y|\leq \frac{L}{R}}\frac{1}{|Y|^{2}}\big(1+|\log R| \nonumber\\
&+|\log dist(\frac{Y}{|Y|},\partial\hat{\Omega})|\big) {\rm d}^{2}Y \nonumber\\
\leq C & \big(|\log R|+|\log R|^{2}\big),\nonumber
\end{align}
where the following integral estimation has been used,
\begin{align}
\int_{4\leq |Y|\leq \frac{L}{R}}&\frac{1}{|Y|^{2}}|\log dist(\frac{Y}{|Y|},\partial\hat{\Omega})|{\rm d}^{2}Y \nonumber\\
&\leq\sum^{|\log_{2}R|}_{n=0}\int_{4\cdot 2^{n}\leq |Y|\leq 4\cdot 2^{n+1} }\frac{1}{|Y|^{2}}|\log dist(\frac{Y}{|Y|},\partial\hat{\Omega})|{\rm d}^{2}Y \nonumber\\
&\leq\sum^{|\log_{2}R|}_{n=0}\int_{4\leq |Z|\leq 8 }\frac{1}{|Z|^{2}}|\log dist(\frac{Z}{|Z|},\frac{\partial\hat{\Omega}}{2^{n}})|{\rm d}^{2}Z \nonumber\\
&\leq \frac{1}{16}\sum^{|\log_{2}R|}_{n=0}\int_{4\leq |Z|\leq 8 }|\log dist(\frac{Z}{|Z|},\frac{\partial\hat{\Omega}}{2^{n}})|{\rm d}^{2}Z \nonumber\\
&\leq C|\log R|.\nonumber
\end{align}
To this extent, we have completed the proof of this lemma.
\end{proof}

\begin{Lemma}\label{L4}
Under the same assumptions of Lemma \ref{L3},
  we obtain
\begin{align}
\Big|\frac{\partial \varphi}{\partial x_{i}}(t,\tilde{x})\Big|\leq C\tilde{x}_{1}\big(1+|\log\tilde{x}_{1}|+|\log\tilde{x}_{1}|^{2}\big),\nonumber
\end{align}
where $C>0$ depends only on $L=\text{diam } \Omega$ and $\|j\|_{\infty}$ and $i=1,2$.
\end{Lemma}
\begin{proof}
We shall prove the lemmea only for  $\frac{\partial \varphi}{\partial x_{2}}(t,\tilde{x})$, since the proof for  $\frac{\partial \varphi}{\partial x_{1}}(t,\tilde{x})$ is similar.
As in the proof of Lemma \ref{L3}, we set the scaled variables $X=\frac{x}{R}, Y=\frac{y}{R}, \tilde{X}=\frac{\tilde{x}}{R}=(1,0)$ and $R=|\tilde{x}|$. In line with the analysis  presented in Lemma \ref{L3}, it is sufficient to show that
$$\int_{\Omega_{R}}|\frac{\partial H}{\partial X_{2}}|{\rm d}Y\leq C\big(1+|\log R|+|\log R|^{2}\big).$$
Two cases are presented for this purpose in the subsequent discussion.

{\bf Case 1.} If $|Y|\geq 2$, we decompose the Greeen function $H(X,Y)=\bar{H}(X,Y)+W(X,Y)$, where
$$\bar{H}(X,Y)=\frac{1}{2\pi}\big(\log|X-Y|-\log|X-Y^{*}|\big),$$
and $Y^{*}$ is the reflection point of $Y$ with respect to the line $\{X_{1}=0\}$.
Furthermore, we restrict $|X|\leq\frac{3}{4}|Y|$, then
\begin{align}
\frac{1}{4}|Y|&\leq |X-Y^{*}|\leq \frac{7}{4}|Y|, \nonumber\\
|\bar{H}(X,Y)|&\leq C\log |X-Y^{*}|\leq C\log |Y|.\nonumber
\end{align}
Due to $0\leq W(X,Y)\leq -\bar{H}(X,Y)$, we have
$$|H(X,Y)|\leq C\log|Y|.$$

On the other hand, for $Y$ fixed, letting $\hat{H}(\xi,\eta)=H(|Y|\xi,|Y|\eta)$, with $\xi=\frac{X}{|Y|},\eta=\frac{Y}{|Y|}$, we consider the restricted region
$\Omega_{0}=\{\xi\in \frac{\Omega_{R}}{|Y|}; |\xi|\leq\frac{3}{4}\}$, then $\hat{H}(\xi,\eta)$ solves the following boundary value problem,
\begin{align}
&\Delta_{\xi}\hat{H}(\xi,\eta)=0,\qquad \xi\in \Omega_{0},\nonumber\\
&|\hat{H}(\xi,\eta)|\leq C\log|Y|,\qquad  \xi\in \partial\Omega_{0}. \nonumber
\end{align}
For any multi-index $\alpha$, the regularity theory yields
$$|\frac{\partial^{\alpha}}{\partial\xi^{\alpha}}\hat{H}(\xi,\eta)|\leq C\log |Y|.$$
Letting $\tilde{\xi}=\frac{\tilde{X}}{|Y|},$ since $|\frac{\partial \hat{H}}{\partial \xi_{2}}(0,\eta)|=0$, by the Taylor theorem we have
$$|\frac{\partial \hat{H}}{\partial \xi_{2}}(\tilde{\xi},\eta)|\leq C\frac{\log |Y|}{|Y|}.$$
Thus, we obtain
$$|\frac{\partial H}{\partial X_{2}}(\tilde{X},Y)|\leq C\frac{\log |Y|}{|Y|^{2}}.$$

{\bf Case 2.} If $|Y|\leq 2$, let $Y_{0} \in \partial\Omega_{R}$ such that $|Y-Y_{0}|=\text{dist} (Y,\partial\Omega_{R})$, and $\bar{Y}$ is the reflection point of $Y$
with respect to  the tangent line at $Y_{0}$.

The Greeen function $H(X,Y)$ is now decomposed as $H(X,Y)=\bar{H}(X,Y)+W(X,Y)$, where
$$\bar{H}(X,Y)=\frac{1}{2\pi}\big(\log|X-Y|-\log|X-\bar{Y}|\big).$$
We consider the  following system:
\begin{align}
&\Delta_{X}W(X,Y)=0,\qquad X,Y\in\Omega_{R},\nonumber\\
&W(X,Y)=-\bar{H}(X,Y),\qquad   X\in\partial\Omega_{R}.\nonumber
\end{align}
Due to $|W(X,Y)|\leq C$ for $X\in\partial \Omega_{R}$, the maximum principle  yields
$|W(X,Y)|\leq C.$

Next we consider the restricted region $\Omega_{R}\cap\{|X|\leq 4\} $, applying the regularity theory, we have
$$|\nabla_{X}W(X,Y)|\leq C.$$
On the other hand, since $|\frac{\partial \bar{H}}{\partial X_{2}}(\tilde{X},Y)|\leq\frac{ C}{|\tilde{X}-Y|}$, we get
$$|\frac{\partial H}{\partial X_{2}}(\tilde{X},Y)|\leq\frac{ C}{|\tilde{X}-Y|}+C.$$
From the above calculations, we deduce that
\begin{align}
\int_{\Omega_{R}}|\frac{\partial H}{\partial X_{2}}(\tilde{X},Y)|{\rm d}Y
&\leq C \int_{|Y|\geq 2}\frac{\log |Y|}{|Y|^{2}}{\rm d}Y+ C \int_{|Y|\leq 2}\frac{1}{|\tilde{X}-Y|^{2}}{\rm d}Y+ C \int_{|Y|\leq 2}{\rm d}Y \nonumber \\
&\leq C +C |\log R|^{2}.\nonumber
\end{align}
\end{proof}

Finally, we give the analogous velocity lemma for the Dirichlet problem.
\begin{Lemma}\label{L5}
For a given constant $\delta >0$, let $\Gamma_{\delta}=([\partial\Omega +B_{\delta}(0)]\cap \Omega)\times \mathbb{R}^{2},$
and $X(s;t,x,v)$, $V(s;t,x,v)$ be the characteristic curves associated with the  relativistic
Vlasov-Poisson system defined previously.
Suppose that $E,\varphi (t,x)$ satisfy the assumptions of  Theorem \ref{T2} and Lemma \ref{L3}, respectively. Then
the existence of solutions to the characteristic equations \eqref{v4.1}-\eqref{v4.3} can be obtained in $[0,T]$
for any $(x,v)\in \bar{\Omega}\times \mathbb{R}^{2}$.
Furthermore,   the following estimate holds  for any $(x, v) \in\Gamma_{\delta}$,   
\begin{align}
C_{1}\Big(v^{2}_{\bot}(0)+x_{\bot}(0)\Big)\leq \Big(v^{2}_{\bot}(t)+x_{\bot}(t)\Big) \leq C_{2}\Big(v^{2}_{\bot}(0)+x_{\bot}(0)\Big), t\in[0,T],\label{v7.5}
\end{align}
for some positive constants $C_{1}, C_{2}$ depending only on
$T,f_{0},\|E\|_{L^{\infty}([0,T],C^{\frac{1}{2}}(\Omega))}, \Omega, \|j\|_{L^{\infty}}$.
\end{Lemma}
\begin{proof}
By the definition of $\alpha$ and $E=\nabla_{x}\varphi(t,x)=E_{l}\cdot U(l)-E_{\bot}n(l)$, taking the derivative of $\alpha$ with respect to $t$ along the characteristics, we get
\begin{align}
\frac{d \alpha}{d t}=&-\sqrt{1+|v|^{2}}\partial_{t}E_{\bot}(t,l,0)x_{\bot}-\frac{\partial \varphi}{\partial t}\nonumber\\
&+\frac{1}{\sqrt{1+|v|^{2}}}\cdot\frac{\omega}{1-\kappa x_{\perp}}\Big(
\big(-\sqrt{1+|v|^{2}}\partial_{l}E_{\bot}(t,l,0)+\frac{d \kappa}{d l}\omega^{2}+\kappa E_{l}\big)x_{\bot}-E_{l}\Big)\nonumber\\
&+\frac{v_{\bot}}{\sqrt{1+|v|^{2}}}\cdot\Big(-\sqrt{1+|v|^{2}}E_{\bot}(t,l,0)+\kappa \omega^{2}-E_{\bot}\Big)\nonumber\\
&+\Big(E_{l}+\frac{v_{\bot}}{\sqrt{1+|v|^{2}}}\frac{\kappa\omega}{1-\kappa x_{\bot}}\Big)\cdot\Big(-\frac{\omega}{\sqrt{1+|v|^{2}}}E_{\bot}(t,l,0)+2\kappa \omega \Big)x_{\bot}\nonumber\\
&+\Big(E_{\bot}-\frac{1}{\sqrt{1+|v|^{2}}}\cdot\frac{\kappa\omega^{2}}{1-\kappa x_{\bot}}\Big)\cdot \Big(v_{\bot}-\frac{v_{\bot}}{\sqrt{1+|v|^{2}}}E_{\bot}(t,l,0)x_{\bot}\Big).\nonumber
\end{align}
With the help of Lemmaa \ref{L3} and \ref{L4}, similarly to the discussion in Lemma \ref{L2}, we conclude that
$$\Big|\frac{d \alpha}{d t}(t,X(t),V(t))\Big|\leq C \alpha \big(1+|\log \alpha|+|\log \alpha|^{2}\big).$$
Then the lemma follows from the Gronwall inequality.
\end{proof}

\section*{Acknowledgments}
Y. Mu was partially supported  by NSFC (Grant No.11701268).
D. Wang was supported in part by NSF grants DMS-2219384 and DMS-2510532.
 
 \medskip

\end{document}